\documentclass[11pt]{article}
\usepackage[T1]{fontenc}
\usepackage[utf8]{inputenc}
\usepackage{lmodern}
\usepackage{microtype}
\usepackage{amsmath,amssymb,amsthm,mathtools,mathrsfs,amscd}
\usepackage[margin=1in]{geometry}
\usepackage[hidelinks]{hyperref}
\usepackage{enumitem}
\usepackage{booktabs}

\newcommand{\CC}{\mathbb C}
\newcommand{\ZZ}{\mathbb Z}
\newcommand{\Zc}{\mathcal Z}
\newcommand{\eps}{\varepsilon}
\newcommand{\wt}{\operatorname{wt}}
\newcommand{\Gr}{\operatorname{Gr}}
\newcommand{\Irr}{\operatorname{Irr}}

\newcommand{\Res}{\operatorname{res}}
\newcommand{\SL}{\operatorname{SL}}
\newcommand{\GL}{\operatorname{GL}}

\newtheorem{conjecture}{Conjecture}
\newtheorem{theorem}[conjecture]{Theorem}
\newtheorem{example}[conjecture]{Example}
\theoremstyle{definition}
\newtheorem{definition}[conjecture]{Definition}
\theoremstyle{plain}
\newtheorem{lemma}[conjecture]{Lemma}
\newtheorem{proposition}[conjecture]{Proposition}
\newtheorem{corollary}[conjecture]{Corollary}
\newtheorem{remark}[conjecture]{Remark}

\title{Levi Filtrations and a Symmetry Conjecture for Mirkovi\'c--Vilonen Bases in Type A}
\author{George Petroulakis\thanks{
Mailing address: Ministry of Education, And. Papandreou 37,
151 80 Athens, Greece.\\
E-mail: \texttt{petroulakis.giorgos@gmail.com}}}
\date{}

\begin{document}
\maketitle

\begin{abstract}
For $G=\mathrm{SL}_r$, we verify a symmetry conjecture of
Mirkovi\'c--Vilonen (MV) bases in two families that extend known type $A_1$ cases. The first is
$V(\omega_p)\otimes V(\omega_q)$ with $1\le p,q\le r-1$, where
$V(\omega_j)\cong\bigwedge^j\mathbb C^r$ is the $j$-th fundamental
representation, in arbitrary rank, and the second is the infinite family of
three-factor tensor products $(\omega_1,\omega_s,\omega_q)$ with
$1\leq s\leq\min\{2,r-1\}$ and $1\leq q\leq r-1$. The latter contains tensor products with irreducible summands of multiplicity greater than one. In the two-factor case the proof characterizes the MV basis uniquely through unitriangularity and compatibility with all the isotypic filtrations associated with the restrictions to standard Levi subgroups. For three factors, we combine a factorization formula for the transition matrix entries with a
separation criterion based on these filtrations. Using the compatibility of
MV bases with Cartan projections, we then show that the symmetry descends
from a sequence of fundamental weights to the sequence of block sums. As an application, we prove the conjecture unconditionally and in arbitrary rank for the two-factor families \((2\omega_1,\omega_q)\), \((\omega_1+\omega_2,\omega_q)\), \((\omega_1,\omega_1+\omega_q)\), and \((\omega_1,\omega_2+\omega_q)\), together with their duals. We also for \(SL_r\), reduce the general conjecture to sequences of fundamental weights. We finally identify a limitation of the method which exhibits the role of multiplicity spaces.
\end{abstract}

\medskip
\noindent\textbf{2020 Mathematics Subject Classification.}
Primary 20G05; Secondary 17B10, 05E10, 14M15.

\smallskip
\noindent\textbf{Keywords and phrases.}
Mirkovi\'c--Vilonen bases, geometric Satake correspondence,
affine Grassmannian, tensor products of representations,
Levi subgroups, crystal bases.

\section{Introduction}
The geometric Satake correspondence realizes the finite-dimensional
representations of a complex connected reductive group $G$ and the
Mirkovi\'c--Vilonen (MV) cycles in the affine Grassmannian yield, via their
fundamental classes, the MV basis of each irreducible representation
\cite{MirkovicVilonen2007}. The MV basis is related to the theory of canonical
bases initiated by Lusztig~\cite{Lusztig1993} where the dual semicanonical basis
provides the related geometric construction~\cite{GeissLeclercSchroer2005}.
\medskip

\noindent In \cite{BGL2024}, the construction was extended to tensor products
and the compatibility of the resulting MV bases with Levi isotypic filtrations
was established. The geometry of MV cycles also admits a polyhedral
description: their moment-map images lead to the theory of MV polytopes, which
provides a combinatorial model for weight and tensor-product multiplicities
\cite{Anderson2003,Kamnitzer2010}. The MV basis has also been studied from the
perspective of perfectness, multiplication, and Duistermaat--Heckman measures
\cite{BaumannKamnitzerKnutson2021}.
\medskip

\noindent For a sequence $\boldsymbol\lambda=(\lambda_1,\ldots,\lambda_n)$ of
dominant weights, we set
\[
 V(\boldsymbol\lambda)=V(\lambda_1)\otimes\cdots\otimes V(\lambda_n).
\]
The MV basis of $V(\boldsymbol\lambda)$ and the tensor product of the
one-factor MV bases are two bases of the same space. If we write
\[
 \langle\!\langle Z''\rangle\!\rangle
   =\sum_{Z\in\Zc(\boldsymbol\lambda)_\mu}
        a^{\boldsymbol\lambda}_{Z,Z''}\langle Z\rangle,
 \qquad Z''\in\Zc(\boldsymbol\lambda)_\mu,
\]
the entries $a^{\boldsymbol\lambda}_{Z,Z''}$ are nonnegative integers and
the transition matrix is unitriangular with diagonal entries one
\cite[Theorem~5.8, Proposition~5.9, and Corollary~5.12]{BGL2024}.
\medskip

\noindent For a dominant weight $\lambda$, let $\lambda^*=-w_0\lambda$, where
$w_0$ is the longest Weyl group element. Then
$\boldsymbol\lambda^*=(\lambda_n^*,\ldots,\lambda_1^*)$ and the
Sch\"utzenberger bijection
$\sigma:\Zc(\boldsymbol\lambda)\to\Zc(\boldsymbol\lambda^*)$ interchanges
crystal operators and reverses the order of the tensor factors, see \cite{HenriquesKamnitzer2006} for the crystal
commutor framework. Now, the symmetry
conjecture we are interested in is:
\begin{conjecture}\cite[Conjecture 5.15]{BGL2024}
\label{conj:BGL}
For every weight $\mu$ and all $Z',Z''\in\Zc(\boldsymbol\lambda)_\mu$,
\[
 a^{\boldsymbol\lambda}_{Z',Z''}
   =a^{\boldsymbol\lambda^*}_{\sigma(Z'),\sigma(Z'')}.
\]
\end{conjecture}

\subsection{Main results and Contribution}

Conjecture~\ref{conj:BGL}  is known in
type $A_1$ \cite{BaumannDemarais2021,BGL2024}. Beyond it, the case
$(\omega_1,\omega_2)$ in $\SL_3$ was examined in
\cite[\S~5.5]{BGL2024}, where the transition matrix is computed.
Although the zero-weight transition matrix has nonzero off-diagonal
entries, the Sch\"utzenberger bijection fixes its three zero-weight
labels, while all other weight spaces are one-dimensional. Thus the
symmetry is direct in this example. Our contribution is that we go beyond these cases by obtaining results
in arbitrary rank, first for pairs of fundamental representations,
then for a family of tensor products with three factors, and
finally for certain non-fundamental highest weights.

\medskip

\noindent\textbf{Two fundamental factors.}
We first consider tensor products
$V(\omega_p)\otimes V(\omega_q)$ for $G=\SL_r(\CC)$ and
$1\leq p,q\leq r-1$. We index the tensor-product basis by pairs
of subsets of $[r]$ of cardinalities $p$ and $q$, respectively,
and use the partial order on these labels induced by the
dominance order on their first components.

\medskip

\noindent The first main result is Lemma~\ref{lem:strict-interval-separation}. It shows that
every strictly comparable pair of labels can be distinguished
by a suitable standard Levi subgroup whose Dynkin diagram is a
finite union of pairwise separated intervals. This property leads to the uniqueness characterization
of the MV basis in Theorem~\ref{thm:two-factor-uniqueness}, where the MV basis is uniquely determined by unitriangularity and
compatibility with all isotypic filtrations.
Together with Hodge duality, this implies the conjecture for
every pair of fundamental representations, in arbitrary rank
and on every weight space in
Corollary~\ref{cor:two-factor-symmetry}. Note that the result is not restricted to weight spaces of small dimension, e.g., the zero-weight space of
\[
  V(\omega_k)\otimes V(\omega_k)
  \quad\text{for }\SL_{2k}
\]
has dimension $\binom{2k}{k}$.

\medskip

\noindent\textbf{Three factors.}
We next consider tensor products with three factors.
For labels $X=(X_1,\ldots,X_n)$ and
$Y=(Y_1,\ldots,Y_n)$ of the same total weight, we use the
partial sums
\[
  \Delta_k(Y,X)
  =\sum_{j=1}^k
     \bigl(\wt(Y_j)-\wt(X_j)\bigr),
  \qquad 1\leq k<n,
\]
\noindent where $\wt(X_j)$ and $\wt(Y_j)$ are the weights of the one-factor MV cycles $X_j,Y_j\in\mathcal Z(\lambda_j)$. 
A necessary condition for the entry $a^{\boldsymbol\lambda}_{Y,X}$ to be nonzero is that
all these partial sums belong to $Q_+=\sum_{i\in\mathcal I}\mathbb Z_{\geq0}\alpha_i$, i.e., the positive root cone, where $\alpha_i$ are the simple roots of $G$. An important step is the factorization formula of
Lemma~\ref{lem:prefix-factorization-final}: whenever one of
the partial sums vanishes, the corresponding matrix entry factors into the coefficients for the two
consecutive blocks. We also establish in
Theorem~\ref{thm:conditional-symmetry} a sufficient condition
for the conjecture which combines the symmetry and
sign properties for proper consecutive blocks
with a condition on the highest weights of crystal
components.

\medskip

\noindent Another key result is that we verify the required filtration condition for the family
\[
  \boldsymbol\lambda=(\omega_1,\omega_s,\omega_q),
  \qquad
  1\leq s\leq\min\{2,r-1\},\quad
  1\leq q\leq r-1,
\]
in Theorem~\ref{thm:three-factor-H}.
The proof is combinatorial, using the Robinson--Schensted
correspondence to describe the highest weights of the relevant
crystal components in terms of shapes of interval words,
and the required comparisons follow from inequalities for
subsequences. The details are given in
Appendix~\ref{app:subsequence-inequalities}.
Combining this result with the factorization formula and the
sufficient condition implies the conjecture for the entire
three-factor family in
Theorem~\ref{thm:three-factor-BGL}. It is important to mention that this three-factor family includes tensor products whose irreducible
summands occur with multiplicity greater than one. For example,
for $\SL_3$,
\[
  V(\omega_1)^{\otimes3}
  \cong V(3\omega_1)
  \oplus 2V(\omega_1+\omega_2)
  \oplus V(0).
\]
So, the point of this three-factor family is that it extends the symmetry
problem to higher rank while including genuinely non-multiplicity-free
tensor products.

\medskip

\noindent\textbf{Non-fundamental highest weights.}
The three-factor symmetry result also yields the conjecture for
two-factor sequences involving non-fundamental highest weights.
More precisely, Theorem~\ref{thm:nonfundamental} establishes
the conjecture for
\[
  (\omega_1+\omega_s,\omega_q),
  \qquad
  (\omega_1,\omega_s+\omega_q),
\]
and for their dual sequences
\[
  (\omega_{r-q},\omega_{r-1}+\omega_{r-s}),
  \qquad
  (\omega_{r-s}+\omega_{r-q},\omega_{r-1}),
\]
in the stated ranges of $r,s,q$.
The first two families follow by passing to Cartan components
of tensor products of fundamental representations and using
the compatibility of MV bases with the corresponding
projections. The dual cases follow from Hodge duality.
This is also important because some of the non-fundamental factors have
weight spaces of dimension greater than one.

\medskip

\noindent Also in Corollary~\ref{cor:reduction} we establish a reduction
of the conjecture for $\SL_r$, where its validity for all finite sequences
of dominant weights follows from its validity for finite sequences of
fundamental weights.

\medskip

\noindent Finally, in Remark~\ref{rem:scope} we explain a limitation of the
sufficient condition used for the three-factor result.
For $\SL_4$ and
$\boldsymbol\lambda=(\omega_1,\omega_3,\omega_2)$, the labels
\[
  Y=(1;123;34),\qquad X=(2;134;13)
\]
satisfy the Levi highest-weight inequalities for every
standard Levi, while
\[
  \Delta_1(Y,X)=\alpha_1,\qquad
  \Delta_2(Y,X)=\alpha_1+\alpha_2+\alpha_3
\]
are both nonzero. This shows that the sufficient condition does not cover arbitrary triples, but also gives a direction for further work in understanding how the conjecture might be established for such pairs.

\subsection{Notation and Organization of the article}
Throughout, $r$ is the rank parameter, $[r]=\{1,\ldots,r\}$, and
$\mathcal I$ is the set of simple-root indices, with simple roots
$\alpha_i$. The following notation is used throughout the
article.
\begin{center}\small
\begin{tabular}{@{}p{0.27\textwidth}p{0.50\textwidth}p{0.17\textwidth}@{}}
\toprule
Symbol & Meaning & Found in\\
\midrule

$A,B$; $C_0,U,D_0$
& The two subsets of a two-factor label $(A,B)$;
$C_0=A\cap B$, $U=A\triangle B$,
$D_0=[r]\setminus(A\cup B)$.
& \S\ref{sec:main-proof}\\

$X=(A,B,C)$; $Y=(A',B',C')$
& Three-factor labels in the family considered in
\S\ref{sec:three-factors};
$|A|=|A'|=1$, $s=|B|=|B'|$,
$q=|C|=|C'|$.
& \S\ref{sec:three-factors}, App.~A\\

$m=|U|$; $m_1,m_2$
& In the two-factor setting, $m=|U|$,
$m_1=p-|C_0|$ and $m_2=q-|C_0|$.
In the descent argument, $m$ is the number of blocks.
& \S\ref{sec:main-proof}; \S\ref{subsec:descent}\\

$\ell_{r_0,s_0}$; $\ell(\lambda)$
& Number of active coordinates
$u_{r_0},\ldots,u_{s_0+1}$ in
$K_{r_0,s_0}$; number of nonzero parts (rows)
of a partition, respectively.
& \S\ref{sec:main-proof}; \S\ref{sec:three-factors}, App.~A\\

$K_{r_0,s_0}$; $J_{r_0,s_0}$
& The coordinate interval and the corresponding
set of simple roots defining the Levi subgroup.
& \S\ref{sec:main-proof}\\

$T$
& Maximal torus.
& \S\ref{sec:prelim}\\

$D_p$; $\mathcal H_{p,q}$; $\mathcal H_{\boldsymbol\lambda}$
& One-factor Hodge map; Hodge map for two factors;
Hodge map for the tensor product associated with
$\boldsymbol\lambda$.
& \S\ref{sec:main-proof}; \S\ref{sec:three-factors}\\

$\operatorname{dec}_Z(I)$; $\operatorname{inc}_Z(I)$
& Length of the longest strictly decreasing subsequence
of the interval word; indicator of the existence of
a weakly increasing triple using one letter from each block.
& \S\ref{sec:three-factors}, App.~A\\

$\Delta_k(Y,X)$; $\Delta_k^{(t)}(Y,X)$
& Partial sum of the first $k$ weight differences;
initial-coordinate sum through $t$ (the coefficient
of $\alpha_t$ for $1\leq t<r$, and zero at $t=0,r$).
& \S\ref{sec:three-factors}, App.~A\\

$I$
& An interval $I=[u,v]\subseteq[r]$ used to define
the interval words.
& \S\ref{sec:three-factors}, App.~A\\

$\operatorname{Id}$
& Identity matrix.
& \S\ref{sec:main-proof}; \S\ref{sec:three-factors}\\

$V_0$; $\mathbf F,\mathbf E,\mathbf E'$
& The vector space $\mathbb C^r$ used in the Hodge
construction; ordered column arrays of tensor-product
and MV bases in the basis-comparison arguments.
& \S\ref{sec:main-proof}\\

$\Omega_h$; $m_h$; $a_{j,h}$; $\mathbf a$
& Coordinates of multiplicity $h$; their number
$m_h=|\Omega_h|$; the number of such coordinates
in the $j$-th tensor factor; and the matrix
$\mathbf a=(a_{j,h})$ of block degrees.
& \S\ref{subsec:levi-multiplicity-spaces}\\

$\mathbf X=(A_1,\ldots,A_n)$
& Tensor-product label consisting of the subsets indexing
its individual factors.
& \S\ref{subsec:levi-multiplicity-spaces}\\

$\mathbf A$; $\mathbf B$; $\mathbf D$
& Transition matrices in the descent argument.
& \S\ref{subsec:descent}\\

$\Lambda$; $\Lambda^p V$
& Character lattice of $T$; $p$-th exterior power
of $V=\mathbb C^r$.
& \S\ref{sec:prelim}; \S\ref{sec:main-proof};
\S\ref{subsec:levi-multiplicity-spaces}\\

$|\boldsymbol\nu|$; $|Z|$; $C_j$; $\|Z\|$
& Sum of the weights in a sequence; Cartan projection
of an MV cycle (with $|Z|=\varnothing$ when the
projection is empty); labels in the $j$-th block
with nonempty Cartan projection; tuple of Cartan
projections, respectively.
& \S\ref{sec:prelim}; \S\ref{subsec:descent}\\

$\eta_J(X)$
& Highest weight of the connected $J$-crystal
component containing the label $X$.
& \S\ref{sec:prelim}; \S\ref{sec:three-factors}\\

\bottomrule
\end{tabular}
\end{center}
The indices $r_0,s_0$ occur only as endpoints of intervals in
\S\ref{sec:main-proof}. In the two-factor uniqueness proof, $\mathbf N$ and $\mathbf N'$
are the transition matrices associated with the two candidate
bases, and $\mathsf U$ is their change-of-basis matrix.
In the higher-factor comparison, $\mathbf M,\mathbf M'$ and
$\mathbf M^*$ denote the transition matrices used in the
symmetry argument.

\medskip
\noindent\textbf{Outline.} Section 2 provides the basic constructions that we will need throughout the paper. Theorem~\ref{thm:two-factor-uniqueness} (uniqueness of the MV basis for two
fundamental factors) and Corollary~\ref{cor:two-factor-symmetry} (symmetry in
that case) are proved in Section~\ref{sec:main-proof}.
Theorem~\ref{thm:three-factor-BGL} (symmetry for
$(\omega_1,\omega_s,\omega_q)$, $s\le 2$) is proved in
Section~\ref{sec:three-factors} and Appendix~\ref{app:subsequence-inequalities}.
Theorem~\ref{thm:nonfundamental} (symmetry with a non-fundamental factor)
and Corollary~\ref{cor:reduction} (reduction to fundamental weights)
are proved in Subsection~\ref{subsec:descent}, using the descent principle
of Proposition~\ref{prop:descent}.

\section{MV bases, Levi restriction, and type $A$ notation}
\label{sec:prelim}
\subsection{MV cycles and transition coefficients}
Let $G$ be connected reductive over $\CC$, $\mathsf B\supset T$ a Borel and
maximal torus, $G^\vee$ the Langlands dual, and
\[
 \mathcal O=\CC[[z]],\qquad \mathcal K=\CC((z)),\qquad
 \Gr=G^\vee(\mathcal K)/G^\vee(\mathcal O).
\]
For $\lambda$ in the coweight lattice, $L_\lambda$ is the point
represented by $z^\lambda$, and for dominant $\lambda$,
$\Gr^\lambda=G^\vee(\mathcal O)\cdot L_\lambda$. If we fix a dominant regular
weight $\theta$, the corresponding cocharacter defines a
$\mathbb C^\times$-action on $\Gr$. For $\mu$ in the coweight lattice,
\[
 \mathscr S_\mu=\left\{x\in\Gr:\lim_{c\to0}\theta(c)\cdot x=L_\mu\right\},
 \qquad
 \mathscr T_\mu=\left\{x\in\Gr:\lim_{c\to\infty}\theta(c)\cdot x=L_\mu\right\}.
\]
\noindent The $n$-fold convolution Grassmannian is $\Gr_n$ with multiplication
\[
 m_n:\Gr_n\longrightarrow\Gr,
 \qquad [g_1,\ldots,g_n]\longmapsto[g_1\cdots g_n].
\]
For $\boldsymbol\lambda=(\lambda_1,\ldots,\lambda_n)$ dominant,
\[
 \Gr_n^{\boldsymbol\lambda}
 =\Gr^{\lambda_1}\widetilde{\times}\cdots
   \widetilde{\times}\Gr^{\lambda_n},
 \qquad
 |\boldsymbol\lambda|=\lambda_1+\cdots+\lambda_n,
\]
with $\Gr_n^{\boldsymbol\lambda}\subset\Gr_n$ locally closed and
$\overline{\Gr_n^{\boldsymbol\lambda}}$ its convolution Schubert
variety. The MV cycles of weight $\mu$ are
\[
 \Zc(\boldsymbol\lambda)_\mu
 =\Irr\left(\overline{\Gr_n^{\boldsymbol\lambda}}
              \cap m_n^{-1}(\mathscr T_\mu)\right),
\]
and their Borel--Moore fundamental classes give the MV basis of
$V(\boldsymbol\lambda)_\mu$ \cite[\S\S~2.2--2.3]{BGL2024}. We write
$\langle Z\rangle$ for the MV basis vector indexed by $Z$ and
$\langle\!\langle Z_1,\ldots,Z_n\rangle\!\rangle$ for the tensor product
of the one-factor MV basis vectors. The factorization bijection
\[
 \Zc(\boldsymbol\lambda)_\mu\cong
 \bigsqcup_{\mu_1+\cdots+\mu_n=\mu}
 \Zc(\lambda_1)_{\mu_1}\times\cdots\times
 \Zc(\lambda_n)_{\mu_n}
\]
labels the tensor-product basis:
\[
 \langle\!\langle Z_1,\ldots,Z_n\rangle\!\rangle
 =\sum_{Z\in\Zc(\boldsymbol\lambda)_\mu}
     a^{\boldsymbol\lambda}_{Z,(Z_1,\ldots,Z_n)}\langle Z\rangle .
\]
\noindent Finally, for an MV cycle $Z\in\mathcal Z(\boldsymbol\lambda)_\mu$,
we write
\[
\wt(Z)=\mu.
\]
For a tensor-product label
$\mathbf b=(Z_1,\ldots,Z_n)$, we set
\[
\wt(\mathbf b)=\sum_{i=1}^n\wt(Z_i).
\]
\subsection{$L$-perfect bases and the Sch\"utzenberger bijection}
\label{subsec:L-perfect}
We recall \cite[\S~3.2]{BGL2024}. Let $\Lambda$ be the character lattice
of $T$, $\Phi$ the root system, $\mathcal I$ the simple roots, and
$Q_+=\sum_{i\in\mathcal I}\mathbb Z_{\geq0}\alpha_i$. For $J\subseteq\mathcal I$ let
$M_J$ be the standard Levi subgroup,
\[
 \Lambda_J^+=\left\{\mu\in\Lambda:
       \langle\alpha_j^\vee,\mu\rangle\geq0\ \ \forall j\in J\right\},
 \qquad
 \mu\leq_J\lambda
 \iff
 \lambda-\mu\in\sum_{j\in J}\mathbb Z_{\geq0}\alpha_j .
\]
For a finite-dimensional rational $G$-module $V$, write
$V=\bigoplus_{\mu\in\Lambda_J^+}V_{J,\mu}$ with $V_{J,\mu}$ the sum of
all $M_J$-subrepresentations isomorphic to $V_J(\mu)$, and
\[
 V_{J,\leq_J\mu}=\bigoplus_{\substack{\nu\in\Lambda_J^+\\
                                         \nu\leq_J\mu}}V_{J,\nu},
 \qquad
 V_{J,<_{J}\mu}=\bigoplus_{\substack{\nu\in\Lambda_J^+\\
                                       \nu<_{J}\mu}}V_{J,\nu},
\]
where $\nu<_{J}\mu$ means $\nu\leq_J\mu$ and $\nu\ne\mu$. The quotient
$V_{J,\leq_J\mu}/V_{J,<_{J}\mu}\cong V_{J,\mu}$ is an $M_J$-module.
\begin{definition}[{\cite[\S~3.2]{BGL2024}}]
\label{def:L-perfect}
A basis $B$ of $V$ is $L$-perfect if for every $J\subseteq\mathcal I$ and every
$\mu\in\Lambda_J^+$:
\begin{itemize}
 \item[(P1)] $V_{J,\leq_J\mu}$ is spanned by a subset of $B$;
 \item[(P2)] the induced basis on
 $V_{J,\leq_J\mu}/V_{J,<_{J}\mu}\cong V_{J,\mu}$ is compatible with a
 decomposition of the isotypic component into irreducibles.
\end{itemize}
\end{definition}
\noindent The MV basis of a tensor product of irreducibles is $L$-perfect
\cite[Theorem~3.4]{BGL2024}.

\begin{lemma}[P1 and $J$-crystal components]
\label{lem:P1-crystal-membership}
Let $\mathcal B$ be the MV basis and let $b\in\mathcal B$. If
$\eta_J(b)$ is the highest weight of the $J$-crystal component containing
$b$, then
\begin{equation}
\label{eq:P1-crystal-membership}
 b\in V_{J,\leq_J\mu}
 \quad\Longleftrightarrow\quad
 \eta_J(b)\leq_J\mu.
\end{equation}
\end{lemma}
\begin{proof}
Let $\mathcal F_{\leq\mu}=V_{J,\leq_J\mu}$ and
$\mathcal F_{<\mu}=V_{J,<_{J}\mu}$. By $P1$,
\[
 \mathcal F_{\leq\mu}=\operatorname{span}(\mathcal B_\mu),
 \qquad \mathcal B_\mu:=\mathcal B\cap\mathcal F_{\leq\mu}.
\]
By \cite[Theorem~3.4]{BGL2024}, the MV basis $\mathcal B$ is $L$-perfect
as a basis of $V$. It remains $L$-perfect when regarded as a basis for
$M_J$, since the standard Levi subgroups of $M_J$ correspond to
subsets $K\subseteq J$, and the corresponding conditions (P1) and (P2)
are inherited from those for $G$. Since $\mathcal F_{\leq\mu}$ is an
$M_J$-submodule, \cite[Lemma~3.1]{BGL2024} implies that $\mathcal B_\mu$ is a
union of connected components of the $J$-crystal. By $P2$ and
\cite[Proposition~3.2]{BGL2024}, each such component is the crystal of an
irreducible $M_J$-summand, with character $\operatorname{ch}V_J(\eta)$
for its highest weight $\eta$. Writing $m_\eta$ for the multiplicity
of $V_J(\eta)$ in $V$, we have
\[
\begin{aligned}
 \operatorname{ch}(\mathcal B_\mu)
 &=\operatorname{ch}(\mathcal F_{\leq\mu})\\
 &=\sum_{\substack{\eta\in\Lambda_J^+\\\eta\leq_J\mu}}
       m_\eta\operatorname{ch}V_J(\eta)\\
 &=\sum_{\substack{\mathcal C\text{ a }J\text{-component}\\
                      \mathcal C\subseteq\mathcal B_\mu}}
       \operatorname{ch}V_J(\eta_J(\mathcal C)).
\end{aligned}
\]
Irreducible $M_J$-characters are linearly independent. On $V$ itself, the character decomposition is
\[
\operatorname{ch}V
 =\sum_{\eta\in\Lambda_J^+}
     m_\eta\operatorname{ch}V_J(\eta).
\]
Thus the $J$-crystal has exactly $m_\eta$ connected components
of highest weight $\eta$. For every $\eta\leq_J\mu$, the preceding
character identity shows that exactly $m_\eta$ such components lie in
$\mathcal B_\mu$. Thus none of these components lies outside
$\mathcal B_\mu$. Conversely, no component of highest weight
$\eta\not\leq_J\mu$ occurs in its character. This proves
\eqref{eq:P1-crystal-membership}.
\end{proof}

\noindent For $\eta$ dominant let $\eta^*=-w_0\eta$. The Sch\"utzenberger bijection
$\sigma:B(\eta)\to B(\eta^*)$ satisfies
\[
 \sigma(\widetilde e_i b)=\widetilde f_i\sigma(b),
 \qquad
 \sigma(\widetilde f_i b)=\widetilde e_i\sigma(b),
 \qquad
 \wt(\sigma(b))=-\wt(b)
 \qquad(i\in\mathcal I),
\]
with the usual convention at $0$. Under the identification of the label sets
\cite[Corollary~4.10]{BGL2024}, the same symbol denotes the induced
bijection on MV cycles, and for a sequence
$
 \boldsymbol\lambda^*=(\lambda_n^*,\ldots,\lambda_1^*)$,
 \[
 \sigma(b_1,\ldots,b_n)
   =(\sigma(b_n),\ldots,\sigma(b_1)).
\]
The crystal structure on $\mathcal Z(\boldsymbol\lambda)$ uses the
opposite tensor-product convention
\cite[Corollary~4.10 and the sentence following it]{BGL2024}; the
Kashiwara tensor rules are as in
\cite[Chapters~4 and~7]{HongKang2002} and \cite{Kashiwara1991}. All
connected components and tensor-factor labels below are understood in
this convention.
For clarity, in type $A_1$ write the one-factor crystal as
$\{1,2\}$ with $\widetilde f(1)=2$. The label $(A,B)$ is the
Kashiwara tensor $B(B)\otimes B(A)$. Thus
\[
  1\otimes1\xrightarrow{\widetilde f}2\otimes1
  \xrightarrow{\widetilde f}2\otimes2,
  \qquad 1\otimes2\text{ is isolated},
\]
which tells us that the Cartan component is
$\{(1,1),(1,2),(2,2)\}$. This is the convention for which the
Cartan condition is $A\preceq B$.
\subsection{The Levi decomposition}
\label{subsec:BGL-Levi}
We recall \cite[\S\S~3.3--3.5]{BGL2024}. Let $J\subseteq\mathcal I$, $M_J$ the Levi subgroup, $\rho_J$ the half-sum of positive coroots in the
subsystem generated by $J$, and for $\eta$ dominant $M_J$-weight let
$\Gr_J^\eta$ be the corresponding $M_J^\vee(\mathcal O)$-orbit in the
affine Grassmannian of the dual Levi. For a sequence $\boldsymbol\lambda$ and $\eta\in\Lambda_J^+$,
we denote by $\mathcal Z^J(\boldsymbol\lambda)_\eta$ the set of
irreducible components of
\[
  \overline{\Gr_n^{\boldsymbol\lambda}}
  \cap(q_{J,\zeta}\circ m_n)^{-1}(\Gr_J^\eta)
\]
whose dimension attains the upper bound
\begin{equation}\label{eq:BGL-dimension-bound}
  2\rho_J(\eta)+\rho_G(|\boldsymbol\lambda|-\eta),
\end{equation}
\cite[\S 3.4, equation (7)]{BGL2024}. Equality holds for every component in
$\Zc^J(\boldsymbol\lambda)_\eta$, not for an arbitrary component of the
fiber. The ordinary MV cycles for $M_J$  are $\Zc_J(\eta)_\mu$ and the
decomposition theorem \cite[\S~3.4]{BGL2024} gives a bijection
\begin{equation}
\label{eq:BGL-bijection}
 \Zc(\boldsymbol\lambda)_\mu
 \;\cong\;
 \bigsqcup_{\eta\in\Lambda_J^+}
       \Zc^J(\boldsymbol\lambda)_\eta
       \times \Zc_J(\eta)_\mu,
\end{equation}
the union over $\eta$ in the appropriate coset modulo the Levi root
lattice, which is a bijection for all MV cycles,
not only for cycles in the fixed locus of a cocharacter.
\medskip

\noindent The proof of \cite[\S~3.5, eqs.~(12)--(13) and Remark~3.5]{BGL2024}
identifies the induced basis on each Levi isotypic subquotient with the
direct sum, over $Y\in\Zc^J(\boldsymbol\lambda)_\eta$, of the MV bases of
the corresponding irreducible $M_J$-modules, and gives
\begin{equation}
\label{eq:number-Y}
 \#\Zc^J(\boldsymbol\lambda)_\eta
 =\operatorname{mult}_{M_J}\!\left(V_{M_J}(\eta),
   \Res_{M_J}^G V(\boldsymbol\lambda)\right).
\end{equation}
With
$\mathcal I_{\boldsymbol\lambda}
   =(m_n)_*\operatorname{IC}(\overline{\Gr_n^{\boldsymbol\lambda}},\mathbb C)$
and the decomposition
\[
 r_J^{\mathcal I}(\mathcal I_{\boldsymbol\lambda})
 =\bigoplus_{\eta\in\Lambda_J^+}
   \operatorname{IC}(\Gr_J^\eta,\mathcal L_\eta)
\]
of \cite[\S~3.5, eq.~(8)]{BGL2024}, the diagram of \cite[\S~3.3]{BGL2024}
identifies $\operatorname{rank}\mathcal L_\eta$
with $\operatorname{mult}_{M_J}(V_{M_J}(\eta),\Res_{M_J}^G
V(\boldsymbol\lambda))$, and from \cite[\S~3.5, eq.~(9)]{BGL2024}, for every
local weight $\mu$ with $\Zc_J(\eta)_\mu\ne\varnothing$ we get
\[
 \operatorname{rank}\mathcal L_\eta\cdot\#\Zc_J(\eta)_\mu
 =\#\Zc^J(\boldsymbol\lambda)_\eta\cdot\#\Zc_J(\eta)_\mu.
\]
Taking $\mu=\eta$, for which $\#\mathcal Z_J(\eta)_\eta=1$, yields
\eqref{eq:number-Y}. The basis compatibility used below is at the
level of Borel--Moore fundamental classes. By the diagram
\cite[\S~3.3]{BGL2024}, restriction to a Levi subgroup is compatible
with tensor products. For each $i$, choose $\eta_i$ such that
\[
\operatorname{mult}_{M_J}\left(
V_{M_J}(\eta_i),\operatorname{Res}_{M_J}^G V(\lambda_i)
\right)=1,
\]
and let $U_i\subset V(\lambda_i)$ be the unique
$M_J$-summand of highest weight $\eta_i$. We fix an
$M_J$-module isomorphism
\[
\iota_i:V_J(\eta_i)\xrightarrow{\sim}U_i
\hookrightarrow V(\lambda_i).
\]
For each $i$, let $Y_i$ be the unique element of
$\mathcal Z^J(\lambda_i)_{\eta_i}$, which exists by
\eqref{eq:number-Y}. For
$Z_i\in\mathcal Z_J(\eta_i)_{\mu_i}$, denote by
$\widetilde Z_i\in\mathcal Z(\lambda_i)_{\mu_i}$ the
MV cycle corresponding to $(Y_i,Z_i)$ under the one-factor
case of \eqref{eq:BGL-bijection}. We write
$\langle Z_i\rangle_J$ for the Borel--Moore fundamental
class defining the MV basis vector of $V_J(\eta_i)$.
The one-factor compatibility of fundamental classes gives
\[
\iota_i\bigl(\langle Z_i\rangle_J\bigr)
=\langle\widetilde Z_i\rangle.
\]
Consequently,
\[
\begin{aligned}
(\iota_1\otimes\cdots\otimes\iota_n)
\left(
\bigotimes_{i=1}^n\langle Z_i\rangle_J
\right)
&=
\bigotimes_{i=1}^n\langle\widetilde Z_i\rangle\\
&=
\langle\!\langle
\widetilde Z_1,\ldots,\widetilde Z_n
\rangle\!\rangle.
\end{aligned}
\]
\subsection{The type $A$ minuscule crystals}
Let $G=\SL_r(\CC)$ and $V=\CC^r$. We identify $V(\omega_k)=\Lambda^k V$ and
its minuscule crystal with the $k$-element subsets of
$[r]=\{1,\ldots,r\}$. A subset $A$ labels a one-factor MV basis vector
$v_A$ of weight $\eps_A=\sum_{i\in A}\eps_i$. For $A\subseteq[r]$, the crystal
operators are
\[
\begin{aligned}
 \widetilde e_i(A)&=
 \begin{cases}
 (A\setminus\{i+1\})\cup\{i\},&i+1\in A,\ i\notin A,\\
 0,&\text{otherwise},
 \end{cases}\\
 \widetilde f_i(A)&=
 \begin{cases}
 (A\setminus\{i\})\cup\{i+1\},&i\in A,\ i+1\notin A,\\
 0,&\text{otherwise}.
 \end{cases}
\end{aligned}
\]
Thus
\[
\begin{aligned}
 (\widetilde e_i A)^c&=\widetilde f_i(A^c),\\
 (\widetilde f_i A)^c&=\widetilde e_i(A^c),
\end{aligned}
\]
so the one-factor Sch\"utzenberger map is complementation
$A\mapsto A^c=[r]\setminus A$. On two-factor labels,
\begin{equation}
\label{eq:sigma-subsets}
 \sigma(A,B)=(B^c,A^c).
\end{equation}
For the $\SL_r$ weight lattice,
\[
 \Lambda=\left(\bigoplus_{i=1}^r\ZZ\eps_i\right)
             /\ZZ(\eps_1+\cdots+\eps_r),
 \qquad \eps_1+\cdots+\eps_r=0,
\]
with $\alpha_i=\eps_i-\eps_{i+1}$ and $\omega_k=\eps_1+\cdots+\eps_k$;
in particular $\omega_1=\eps_1$ and $\omega_{r-1}=-\eps_r$.
\section{Two-factor tensor products of fundamental representations}
\label{sec:main-proof}

\subsection{Weight labels and their order}
Fix $1\leq p,q\leq r-1$ and a weight $\nu$ of
$V(\omega_p)\otimes V(\omega_q)$. A tensor-product label is a pair
$(A,B)$ with $|A|=p$, $|B|=q$. We set
\[
 C_0=A\cap B,\qquad U=A\mathbin\triangle B,
 \qquad D_0=[r]\setminus(A\cup B),
\]
and write
\[
 c_0=|C_0|,\quad m=|U|,\quad d=|D_0|,
 \quad m_1=p-c_0,\quad m_2=q-c_0.
\]
Then $m=m_1+m_2$ and $r=c_0+m+d$. Since the sum of the coordinates of
$\eps_A+\eps_B$ is $p+q$, two such multiplicity profiles define the
same $\SL_r$-weight only if they are equal. Hence the labels of weight
$\nu$ are
\begin{equation}
\label{eq:labels-XP}
 X_P=(C_0\cup P,\ C_0\cup(U\setminus P)),
 \qquad P\in\binom U{m_1}.
\end{equation}
Let $U=\{u_1<\cdots<u_m\}$. We define
\begin{equation}
\label{eq:Bruhat-order-P}
 P'\preceq P
 \quad\Longleftrightarrow\quad
 \delta_t(P',P):=
 |P'\cap\{u_1,\ldots,u_t\}|
 -|P\cap\{u_1,\ldots,u_t\}|\geq0
 \quad(0\leq t\leq m).
\end{equation}
Here $\delta_0=\delta_m=0$. For a vector
$\gamma=\sum_{i=1}^r \gamma_i\eps_i$ with $\sum_i \gamma_i=0$,
\[
 \gamma=\sum_{k=1}^{r-1}\left(\sum_{i=1}^k \gamma_i\right)\alpha_k.
\]
For $k=0,\ldots,r$, let $t(k)=\#\{j:u_j\leq k\}$. The
partial sum of the coordinate difference through $k$ is exactly
$\delta_{t(k)}(P',P)$. Therefore
\begin{equation}
\label{eq:root-cone-subset-order}
 P'\preceq P
 \quad\Longleftrightarrow\quad
 \eps_{C_0\cup P'}-\eps_{C_0\cup P}\in Q_+.
\end{equation}

\noindent We also write $f_P=v_{C_0\cup P}\otimes v_{C_0\cup(U\setminus P)}$ for the
tensor-product basis vector. We index the MV basis of the same weight space
by $P$ through the identification
$\mathcal Z(\omega_p,\omega_q)\simeq
\mathcal Z(\omega_p)\otimes\mathcal Z(\omega_q)$
\cite[Corollary~4.10]{BGL2024}, and denote its
vectors by $e_P$. Then, by \cite[Corollary~5.12 and Proposition~5.9]{BGL2024} we have
\begin{equation}
\label{eq:MV-unitriangular-support}
 f_P=\sum_{P'\preceq P}a_{P',P}e_{P'},
 \qquad a_{P,P}=1.
\end{equation}
Indeed, if $a_{P',P}\ne0$, $\eps_{C_0\cup P'}-\eps_{C_0\cup P}\in Q_+$, which is
\eqref{eq:root-cone-subset-order}. In particular, for any linear
extension $P_1,\ldots,P_N$ of $\preceq$,
\begin{equation}
\label{eq:common-flag}
 \operatorname{span}(f_{P_1},\ldots,f_{P_k})
 =\operatorname{span}(e_{P_1},\ldots,e_{P_k})
 \quad(1\leq k\leq N).
\end{equation}

\subsection{The Levi component invariant}
For $1\leq r_0\leq s_0\leq m-1$, let
\[
 K_{r_0,s_0}=[u_{r_0},u_{s_0+1}],
 \qquad
 J_{r_0,s_0}=\{\alpha_i:u_{r_0}\leq i<u_{s_0+1}\}.
\]
Thus $J_{r_0,s_0}$ is the set of simple roots
corresponding to the interval $K_{r_0,s_0}$, and the
associated Levi subgroup $M_{J_{r_0,s_0}}$ has its
nontrivial block on this interval. We now set
\[
 \ell_{r_0,s_0}=s_0-r_0+2,
 \quad c_{r_0,s_0}=|C_0\cap K_{r_0,s_0}|,
 \quad d_{r_0,s_0}=|D_0\cap K_{r_0,s_0}|,
 \quad n_{r_0,s_0}=c_{r_0,s_0}+\ell_{r_0,s_0}+d_{r_0,s_0}.
\]
For $P\in\binom U{m_1}$, write
\[
 A_K=(C_0\cap K_{r_0,s_0})\cup
       (P\cap\{u_{r_0},\ldots,u_{s_0+1}\}),
 \qquad
 B_K=(C_0\cap K_{r_0,s_0})\cup
       ((U\setminus P)\cap\{u_{r_0},\ldots,u_{s_0+1}\}).
\]
We identify the coordinates in $K_{r_0,s_0}$, in their natural
increasing order, with $1,\ldots,n_{r_0,s_0}$ and we define
\begin{equation}
\label{eq:g-and-M}
 g_h(A_K,B_K)=|A_K\cap[1,h]|+|B_K\cap\{h+1,\ldots,n_{r_0,s_0}\}|,
 \qquad
 \kappa_K(A_K,B_K)=\min_{0\leq h\leq n_{r_0,s_0}}g_h(A_K,B_K).
\end{equation}

\begin{lemma}[Crystal invariance of $\kappa_K$]
\label{lem:M-crystal-invariant}
With the opposite tensor-product convention for the crystals
$\mathcal Z(\boldsymbol\lambda)$, specified immediately after
\cite[Corollary~4.10]{BGL2024}, the pair $(A_K,B_K)$ is the
Kashiwara tensor $B(B_K)\otimes B(A_K)$. The function $\kappa_K(A_K,B_K)$ is constant
on each $J_{r_0,s_0}$-crystal component.
\end{lemma}
\begin{proof}
For $S\subseteq[n_{r_0,s_0}]$, set
\[
 \varepsilon_i(S)=\mathbf1_{\{i+1\in S,\ i\notin S\}},
 \qquad
 \varphi_i(S)=\mathbf1_{\{i\in S,\ i+1\notin S\}}.
\]
The tensor rules on $B\otimes A$ are
\[
 \widetilde f_i(B\otimes A)=
 \begin{cases}
  \widetilde f_iB\otimes A,&\varphi_i(B)>\varepsilon_i(A),\\
  B\otimes\widetilde f_iA,&\varphi_i(B)\leq\varepsilon_i(A),
 \end{cases}
\]
\[
 \widetilde e_i(B\otimes A)=
 \begin{cases}
  \widetilde e_iB\otimes A,&\varphi_i(B)\geq\varepsilon_i(A),\\
  B\otimes\widetilde e_iA,&\varphi_i(B)<\varepsilon_i(A).
 \end{cases}
\]
If we write $a_j=\mathbf1_{\{j\in A\}}$ and
$b_j=\mathbf1_{\{j\in B\}}$, the only value of $g_h$ changed by
$\widetilde f_i$ is $g_i$, and
\begin{equation}
\label{eq:g-neighbor-differences}
 g_{i-1}-g_i=b_i-a_i,
 \qquad
g_{i+1}-g_i=a_{i+1}-b_{i+1}.
\end{equation}
If $\widetilde f_i$ acts on $A$, then
\[
 (a_i,a_{i+1})=(1,0),\qquad
 \varepsilon_i(A)=0,\qquad \varphi_i(B)=0.
\]
Thus $g_i$ decreases by $1$. If $b_i=0$, then
$g_{i-1}=g_i-1$, and if $b_i=1$, the equality $\varphi_i(B)=0$
forces $b_{i+1}=1$, so $g_{i+1}=g_i-1$. In either case the old
minimum is attained at an index other than $i$, with value at most
$g_i-1$; decreasing $g_i$ by $1$ does not change the minimum.

\medskip

\noindent If $\widetilde f_i$ acts on $B$, then
\[
 (b_i,b_{i+1})=(1,0),\qquad
 \varphi_i(B)=1>\varepsilon_i(A)=0.
\]
Thus $g_i$ increases by $1$. If $a_{i+1}=0$, then
$g_{i+1}=g_i$, and if $a_{i+1}=1$, the equality
$\varepsilon_i(A)=0$ forces $a_i=1$, so $g_{i-1}=g_i$.
Again the minimum is unchanged, and since $\widetilde e_i$ reverses
$\widetilde f_i$, $\kappa_K$ is constant on every crystal component.
\end{proof}

\begin{lemma}[Highest-weight label]
\label{lem:highest-weight-label}
Let $A,B\subseteq[n]$, with $|A|=a_0$ and $|B|=b_0$, and regard
$(A,B)$ as $B(B)\otimes B(A)$. If $(A,B)$ is highest weight, then for
a unique integer $j$,
\[
 B=[1,b_0],
 \qquad
 A=[1,j]\cup[b_0+1,b_0+a_0-j],
 \qquad
 j=|A\cap B|=\kappa(A,B).
\]
Its $\GL_n$ highest weight is
\begin{equation}
\label{eq:GLn-highest-partition}
 (2^j,1^{a_0+b_0-2j},0^{n-a_0-b_0+j}).
\end{equation}
\end{lemma}
\begin{proof}
The tensor rule for $\widetilde e_i$ gives
\[
 \widetilde e_i(B\otimes A)=0\quad\forall i
 \quad\Longleftrightarrow\quad
 \varepsilon_i(B)=0\quad\text{and}\quad
 \varepsilon_i(A)\leq\varphi_i(B)\quad\forall i.
\]
The first condition says that the binary word of $B$ has no adjacent
pattern $01$, hence $B=[1,b_0]$. Since
$\varphi_i(B)=1$ only for $i=b_0$, the second condition says that the
word of $A$ has no adjacent pattern $01$ except possibly at $b_0$.
Consequently
\[
 A=[1,j]\cup[b_0+1,b_0+a_0-j]
\]
for some $j$, and $A\cap B=[1,j]$. For $0\leq h\leq b_0$,
\[
 g_h=\min(j,h)+b_0-h\geq j,
\]
while for $b_0\leq h\leq n$,
\[
 g_h=j+\min(h-b_0,a_0-j)\geq j.
\]
Since $g_{b_0}=j$, we obtain $\kappa(A,B)=j$. The coordinate
multiplicities of the highest weight are $2$ on $[1,j]$, $1$ on the
next $a_0+b_0-2j$ positions, and $0$ on the remaining
$n-a_0-b_0+j$ positions, which implies \eqref{eq:GLn-highest-partition}.
\end{proof}

\begin{proposition}[Highest weight]
\label{prop:interval-Levi-label}
For $P\in\binom U{m_1}$, the highest weight of the
$J_{r_0,s_0}$-crystal component containing $X_P$ has, on
$K_{r_0,s_0}$, the dominant $\GL_{n_{r_0,s_0}}$-weight
\begin{equation}
\label{eq:interval-component-weight}
 \eta_K(P)=
 (2^{c+j_K(P)},1^{\ell-2j_K(P)},0^{d_K+j_K(P)}),
\end{equation}
where, for the word $w_1\cdots w_\ell$ defined by
\[
 w_t=\begin{cases}2,&u_{r_0+t-1}\in P,\\
                    1,&u_{r_0+t-1}\notin P,\end{cases}
\]
we let
\begin{equation}
\label{eq:jK}
 j_K(P)=\min_{0\leq h\leq\ell}
 \left(\#\{2\text{ among }w_1,\ldots,w_h\}
       +\#\{1\text{ among }w_{h+1},\ldots,w_\ell\}\right),
\end{equation}
where $c=c_{r_0,s_0}$, $d_K=d_{r_0,s_0}$, and
$\ell=\ell_{r_0,s_0}$. The $J_{r_0,s_0}$-dominant weight
$\eta_{J_{r_0,s_0}}(P)$ is obtained by adjoining the unchanged
coordinates of $\nu$ outside $K_{r_0,s_0}$.
\end{proposition}
\begin{proof}
For every cut $h$ in the interval $K_{r_0,s_0}$, each coordinate
of $C_0\cap K_{r_0,s_0}$ contributes exactly one to $g_h$, and each
coordinate of $D_0\cap K_{r_0,s_0}$ contributes zero. A coordinate
$u_t\in P$ contributes one precisely when it is to the left of the
cut; a coordinate $u_t\notin P$ contributes one precisely when it is
to the right. Therefore
\[
 \kappa_K(A_K,B_K)=c+j_K(P).
\]
The two local tensor factors have degrees
\[
 |A_K|=c+|P\cap K|,
 \qquad |B_K|=c+\ell-|P\cap K|,
 \qquad |A_K|+|B_K|=2c+\ell.
\]
By Lemmas~\ref{lem:M-crystal-invariant} and
\ref{lem:highest-weight-label}, the highest weight of the component
is obtained by substituting $j=\kappa_K(A_K,B_K)=c+j_K(P)$ and total
degree $2c+\ell$ into \eqref{eq:GLn-highest-partition}. Since
$n_{r_0,s_0}=c+\ell+d_K$, this gives \eqref{eq:interval-component-weight}.
\end{proof}

\subsection{Strict separation by a union of Levi subgroups}
\begin{lemma}[Strict separation]
\label{lem:strict-interval-separation}
If $P'\prec P$, let
\[
 \mathcal R=\{t\in\{1,\ldots,m-1\}:\delta_t(P',P)>0\}
   =\bigsqcup_{\ell=1}^k[r_\ell,s_\ell]
\]
be its decomposition into maximal intervals. For each $\ell$ set
\[
 K_\ell=[u_{r_\ell},u_{s_\ell+1}],
 \qquad
 J_\ell=\{\alpha_i:u_{r_\ell}\leq i<u_{s_\ell+1}\},
 \qquad J=\bigcup_{\ell=1}^kJ_\ell.
\]
Then $J$ is a standard Levi subgroup and
\begin{equation}
\label{eq:strict-j-separation}
 \eta_J(P')-\eta_J(P)\in Q_J^+\setminus\{0\}.
\end{equation}
\end{lemma}
\begin{proof}
For each $\ell$, maximality gives
\[
 \delta_{r_\ell-1}=0,
 \quad \delta_{r_\ell},\ldots,\delta_{s_\ell}>0,
 \quad \delta_{s_\ell+1}=0,
\]
and in particular, $r_{\ell+1}\geq s_\ell+2$, so the coordinate
intervals $K_\ell$ are pairwise disjoint and separated by at least
one omitted simple root. Thus $J$ is a standard Levi with connected
components $J_1,\ldots,J_k$. The $J$-crystal is the product of the
crystals for these mutually commuting root subsystems. Therefore
\[
 \eta_J(P)|_{K_\ell}=\eta_{J_\ell}(P)|_{K_\ell},
 \qquad
 \eta_J(P)|_{[r]\setminus\bigcup_\ell K_\ell}
       =\nu|_{[r]\setminus\bigcup_\ell K_\ell}.
\]

\noindent We now fix $\ell$, write $L_\ell=s_\ell-r_\ell+2$, and let
$x=w(P')|_{K_\ell}$ and $y=w(P)|_{K_\ell}$ be the active words,
with $2$ denoting membership in the first subset. Their contents agree,
and
\[
 x_1=2,\quad y_1=1,\qquad x_{L_\ell}=1,\quad y_{L_\ell}=2.
\]
For a word $z=z_1\cdots z_{L_\ell}$ let
\[
 \psi_h(z)=\#\{2\text{ in }z_1,\ldots,z_h\}
           +\#\{1\text{ in }z_{h+1},\ldots,z_{L_\ell}\}.
\]
For $1\leq h<L_\ell$, equality of the total contents gives
\[
 \psi_h(x)-\psi_h(y)=2\delta_{r_\ell+h-1}>0,
 \qquad
 \psi_0(x)=\psi_0(y),\quad\psi_{L_\ell}(x)=\psi_{L_\ell}(y).
\]
Moreover,
\[
 \psi_1(y)=\psi_0(y)-1,
 \qquad
 \psi_{L_\ell-1}(y)=\psi_{L_\ell}(y)-1,
\]
and hence
\[
\min_h\psi_h(y)
\leq
\min\{\psi_1(y),\psi_{L_\ell-1}(y)\}
<
\min\{\psi_0(y),\psi_{L_\ell}(y)\}.
\]
The endpoint equalities and the interior inequalities give
\[
\begin{aligned}
\psi_0(x)&=\psi_0(y)>\min_h\psi_h(y),\\
\psi_{L_\ell}(x)&=\psi_{L_\ell}(y)>\min_h\psi_h(y),\\
\psi_h(x)&>\psi_h(y)\geq\min_k\psi_k(y),
\qquad 1\leq h\leq L_\ell-1.
\end{aligned}
\]
Therefore
\[
\min_h\psi_h(x)>\min_h\psi_h(y),
\]
and thus
\[
j_{K_\ell}(P')=\min_h\psi_h(x)
>
\min_h\psi_h(y)=j_{K_\ell}(P).
\]

\noindent By Proposition~\ref{prop:interval-Levi-label}, for each
$\ell$ set
\[
j_\ell=j_{K_\ell}(P),
\qquad
j'_\ell=j_{K_\ell}(P'),
\qquad j'_\ell>j_\ell.
\]
The restrictions of $\eta_J(P)$ and $\eta_J(P')$ to $K_\ell$
are respectively
\[
\lambda_{\ell,j_\ell}=(2^{c_\ell+j_\ell},
                    1^{L_\ell-2j_\ell},
                    0^{d_\ell+j_\ell}),
\qquad
\lambda_{\ell,j'_\ell}=(2^{c_\ell+j'_\ell},
                    1^{L_\ell-2j'_\ell},
                    0^{d_\ell+j'_\ell}),
\]
where
\[
c_\ell=|C_0\cap K_\ell|,
\qquad
d_\ell=|D_0\cap K_\ell|.
\]
In local coordinates on $K_\ell$, we have
\[
\lambda_{\ell,q+1}-\lambda_{\ell,q}
=\varepsilon_{c_\ell+q+1}
 -\varepsilon_{c_\ell+L_\ell-q}
=\sum_{t=c_\ell+q+1}^{c_\ell+L_\ell-q-1}
       \alpha_{u_{r_\ell}+t-1}.
\]
Hence
\[
\begin{aligned}
\lambda_{\ell,j'_\ell}-\lambda_{\ell,j_\ell}
&=\sum_{q=j_\ell}^{j'_\ell-1}
  (\lambda_{\ell,q+1}-\lambda_{\ell,q})\\
&=\sum_{q=j_\ell}^{j'_\ell-1}
  \sum_{t=c_\ell+q+1}^{c_\ell+L_\ell-q-1}
       \alpha_{u_{r_\ell}+t-1}\\
&\in Q_{J_\ell}^+\setminus\{0\}.
\end{aligned}
\]
For $u_t\notin\bigcup_\ell K_\ell$, we have
\[
\delta_t=\delta_{t-1}=0
\quad\Longrightarrow\quad
\mathbf 1_{P'}(u_t)=\mathbf 1_P(u_t)
\quad\Longrightarrow\quad
\bigl(\eta_J(P')\bigr)_{u_t}
=\bigl(\eta_J(P)\bigr)_{u_t}.
\]
Consequently, in the global coordinates,
\[
\begin{aligned}
\eta_J(P')-\eta_J(P)
&=\sum_\ell
  \bigl(\lambda_{\ell,j'_\ell}-\lambda_{\ell,j_\ell}\bigr)\\
&=\sum_\ell\sum_{q=j_\ell}^{j'_\ell-1}
  \sum_{t=c_\ell+q+1}^{c_\ell+L_\ell-q-1}
       \alpha_{u_{r_\ell}+t-1}\\
&\in Q_J^+\setminus\{0\},
\end{aligned}
\]
since the $J_\ell$ have pairwise disjoint simple-root supports.
This proves \eqref{eq:strict-j-separation}.
\end{proof}

\begin{corollary}[Separation of labels]
\label{cor:interval-Levi-separates}
If $P'\prec P$, there is a Levi subgroup $M_J$ (with $J$ a finite
union of pairwise separated intervals) such that
\[
 \eta_J(P')-\eta_J(P)\in Q_J^+\setminus\{0\}.
\]
For $\mu=\eta_J(P)$, the subspace
$V_{J,\leq_J\mu}\cap V_\nu$ therefore satisfies
\begin{equation}
\label{eq:separating-P1-space}
 e_P\in V_{J,\leq_J\mu}\cap V_\nu,
 \qquad
 e_{P'}\notin V_{J,\leq_J\mu}\cap V_\nu.
\end{equation}
\end{corollary}
\begin{proof}
The weight difference assertion is Lemma~\ref{lem:strict-interval-separation} and the membership statement follows from Lemma~\ref{lem:P1-crystal-membership}.
\end{proof}

\begin{theorem}[Uniqueness of the MV basis]
\label{thm:two-factor-uniqueness}
Let $G=\SL_r(\CC)$, $r\geq2$, and $1\leq p,q\leq r-1$. For a
weight $\nu$, we use the indexing $X_P$ of \eqref{eq:labels-XP}. The MV
basis $\{e_P:P\in\binom U{m_1}\}$ of
$(V(\omega_p)\otimes V(\omega_q))_\nu$ is the unique basis
satisfying:
\begin{enumerate}[label=\textup{(\roman*)}]
  \item
  \[
    f_P=\sum_{P'\preceq P}c_{P',P}e_{P'},
    \qquad c_{P,P}=1,
  \]
  \item for every standard Levi subgroup $M_J$ and every
  $\mu\in\Lambda_J^+$, the subspace
  \[
    V_{J,\leq_J\mu}\cap
    (V(\omega_p)\otimes V(\omega_q))_\nu
  \]
  is spanned by a subset of $\{e_P:P\in\binom U{m_1}\}$.
\end{enumerate}
\end{theorem}

\begin{proof}
Let $\mathcal E_\nu=\{e_P:P\in\binom{U}{m_1}\}$ be the MV basis and
let $\mathcal E'_\nu=\{e'_P:P\in\binom{U}{m_1}\}$ be another basis
satisfying (i)--(ii). We choose a
linear extension $P_1,\ldots,P_N$ of $\preceq$ and set
\[
 F_k=\operatorname{span}(f_{P_1},\ldots,f_{P_k}),
 \qquad F_0=0.
\]
Unitriangularity gives
\[
 F_k=\operatorname{span}(e_{P_1},\ldots,e_{P_k})
    =\operatorname{span}(e'_{P_1},\ldots,e'_{P_k}).
\]
If a subspace $S$ is spanned by a subset of $\mathcal E$, then
\[
 \mathbf1_{\{e_{P_k}\in S\}}
 =\dim(S\cap F_k)-\dim(S\cap F_{k-1}),
\]
and the same holds for $\mathcal E'$. By (ii), the set of basis
indices spanning each $V_{J,\leq_J\mu}\cap V_\nu$ is therefore the
same for $\mathcal E$ and $\mathcal E'$. Hence the change-of-basis
matrix $\mathsf U$ defined by $e'=e\mathsf U$ preserves every such subspace.

\noindent Let $P'\prec P$. By Corollary~\ref{cor:interval-Levi-separates},
there is a standard Levi subspace $S$ with $e_P\in S$ and
$e_{P'}\notin S$. Since $\mathsf U(S)=S$, $e'_P=\mathsf Ue_P\in S$ and the
coefficient of $e_{P'}$ in $e'_P$ is therefore zero.

\medskip

\noindent For completeness, write the two transition matrices as
$\mathbf F=\mathbf E\mathbf N=\mathbf E'\mathbf N'$, where
$\mathbf F,\mathbf E,\mathbf E'$ are the ordered column arrays of the bases
$f,e,e'$. Then
\[
 \mathsf U=\mathbf N\mathbf N'^{-1}.
\]
But if $
\mathbf R=\mathbf N'-\operatorname{Id},
$
we have
\[\mathbf R_{P',P}=0\quad\text{unless }P'\prec P,
\qquad
\mathbf R^{d}=0,
\quad d=\#\binom{U}{m_1}.
\]
Hence,
\[
\mathbf N'^{-1}
=(\operatorname{Id}+\mathbf R)^{-1}
=\sum_{k=0}^{d-1}(-1)^k\mathbf R^k.
\]
For $k\geq1$,
\[
(\mathbf R^k)_{P',P}
=
\sum_{P'=P_0\prec P_1\prec\cdots\prec P_k=P}
\prod_{a=1}^k
\mathbf R_{P_{a-1},P_a}.
\]
Thus
\[
(\mathbf N'^{-1})_{P',P}=0
\qquad\text{unless }P'\preceq P.
\]
It follows from $\mathsf U=\mathbf N\mathbf N'^{-1}$ that
\[
\mathsf U_{P',P}
=\sum_{Q\in\binom{U}{m_1}}
\mathbf N_{P',Q}(\mathbf N'^{-1})_{Q,P}
=0
\qquad\text{unless }P'\preceq P.
\]
Moreover,
\[
\mathsf U_{P,P}
=\mathbf N_{P,P}(\mathbf N'^{-1})_{P,P}=1,
\]
and the preceding argument gives
\[
\mathsf U_{P',P}=0
\qquad\text{for every }P'\prec P.
\]
Therefore
\[
\mathsf U_{P',P}=\delta_{P',P},
\qquad
\mathsf U=\operatorname{Id},
\qquad
\mathcal E'=\mathcal E.
\]
\end{proof}
\noindent The following example, for $\SL_4$, $(\omega_2,\omega_2)$, and zero weight, shows how to determine the six MV
basis vectors explicitly. It gives the unitriangular expansion
of the tensor-product basis and verifies that every 
filtration subspace satisfying condition (P1) is spanned by a subset of these vectors.
\begin{example}
\label{ex:SL4-six-dimensional}

Let $V_0=\CC^4$ with ordered basis $u_1,\ldots,u_4$.
We identify $V(\omega_2)=\Lambda^2V_0$. For
$A=\{a_1<a_2\}\in\binom{[4]}2$, let
\[
u_A=u_{a_1}\wedge u_{a_2}.
\]
We use the exterior-power normalization
$v_A=u_A$ for the one-factor MV basis. By \cite[Theorem~5.2]{BaumannKamnitzerKnutson2021},
the MV basis is perfect. Since $V(\omega_2)$ is minuscule, every nonzero
crystal string for a simple root has length one. Thus, if
$\widetilde f_i(A)=A'$, then $\varepsilon_i(A)=0$ and
$\varepsilon_i(A')=1$. Perfectness gives
$e_i v_{A'}=v_A$ and $e_i v_A=0$. Writing
$f_i v_A=c v_{A'}$, the relation $[e_i,f_i]=h_i$ and
$h_i v_A=v_A$ imply $c=1$. In the exterior-power model, with Chevalley operators
$e_i=E_{i,i+1}$ and $f_i=E_{i+1,i}$, the operator $f_i$ sends
$u_A$ to $u_{A'}$ with coefficient $+1$.
Starting with the highest vector $v_{\{1,2\}}=u_1\wedge u_2$,
induction along the connected minuscule crystal therefore gives
$v_A=u_A$ for every $A$. The minuscule crystal $B(\omega_2)$ is indexed by
$\binom{[4]}2$, with $\operatorname{wt}(A)=\eps_A$.
Consequently,
\[
\begin{aligned}
\eps_A+\eps_B=0
&\iff \mathbf 1_A+\mathbf 1_B=(1,1,1,1)\\
&\iff B=A^c,
\end{aligned}
\]
where the first equivalence uses the
$\SL_4$ weight lattice and $|A|=|B|=2$.
Thus
\[
\mathcal X_0
=\{X_A=(A,A^c):A\in\binom{[4]}2\}.
\]
The corresponding tensor-product basis vectors are
\[
f_A:=f_{X_A}=v_A\otimes v_{A^c}.
\]
We order them as
\[
\begin{gathered}
f_1=f_{12},\qquad f_2=f_{13},\qquad
f_3=f_{14},\\
f_4=f_{23},\qquad f_5=f_{24},\qquad
f_6=f_{34},
\end{gathered}
\]
and set
\[
\mathcal W=
\bigl(V(\omega_2)\otimes V(\omega_2)\bigr)_0
=
\bigl(\Lambda^2V_0\otimes\Lambda^2V_0\bigr)_0.
\]
We see that $\dim\mathcal W=6$. For $i=1,2,3$, let
\[
x_i=E_{i,i+1},\qquad
y_i=E_{i+1,i},\qquad
h_i=[x_i,y_i],
\]
acting through the tensor-product representation, and define
\[
\mathcal C_i=h_i^2+2h_i+4y_ix_i.
\]
On an irreducible $\mathfrak{sl}_2$-module of highest weight
$m$, $\mathcal C_i$ acts by $m(m+2)$. Hence its kernel is
the trivial isotypic component. In the ordered basis
$\{f_1,\ldots,f_6\}$, the restrictions of these operators to
$\mathcal W$ are
\[
\mathcal C_1=
\begin{pmatrix}
0&0&0&0&0&0\\
0&4&0&4&0&0\\
0&0&4&0&4&0\\
0&4&0&4&0&0\\
0&0&4&0&4&0\\
0&0&0&0&0&0
\end{pmatrix},
\qquad
\mathcal C_2=
\begin{pmatrix}
4&4&0&0&0&0\\
4&4&0&0&0&0\\
0&0&0&0&0&0\\
0&0&0&0&0&0\\
0&0&0&0&4&4\\
0&0&0&0&4&4
\end{pmatrix},
\]
\[
\mathcal C_3=
\begin{pmatrix}
0&0&0&0&0&0\\
0&4&4&0&0&0\\
0&4&4&0&0&0\\
0&0&0&4&4&0\\
0&0&0&4&4&0\\
0&0&0&0&0&0
\end{pmatrix}.
\]

If we write
\[
W_{J,\mu}:=V_{J,\leq_J\mu}\cap\mathcal W,
\]
then for $J=\{i\}$ and $\mu=0$, the matrices give
\[
\begin{aligned}
W_{\{1\},0}
&=\ker\mathcal C_1
 =\operatorname{span}\{f_1,f_4-f_2,f_5-f_3,f_6\},\\
W_{\{2\},0}
&=\ker\mathcal C_2
 =\operatorname{span}\{f_2-f_1,f_3,f_4,f_6-f_5\},\\
W_{\{3\},0}
&=\ker\mathcal C_3
 =\operatorname{span}\{f_1,f_3-f_2,f_5-f_4,f_6\}.
\end{aligned}
\]
The only highest weights of the corresponding
$M_{\{i\}}$-summands meeting $\mathcal W$ are $0$ and
$\alpha_i$. Thus
\[
W_{\{i\},\alpha_i}=\mathcal W.
\]

\noindent For the connected rank-two Levi subgroups, intersection of
the corresponding rank-one invariant subspaces gives
\[
\begin{aligned}
W_{\{1,2\},0}
&=W_{\{1\},0}\cap W_{\{2\},0}\\
&=\operatorname{span}\{f_1-f_2+f_4,\,
                         f_3-f_5+f_6\},\\
W_{\{2,3\},0}
&=W_{\{2\},0}\cap W_{\{3\},0}\\
&=\operatorname{span}\{f_1-f_2+f_3,\,
                         f_4-f_5+f_6\}.
\end{aligned}
\]
The only highest weights of the corresponding Levi summands
meeting $\mathcal W$ are respectively
$0,\alpha_1+\alpha_2$ and $0,\alpha_2+\alpha_3$.
Thus,
\[
W_{\{1,2\},\alpha_1+\alpha_2}
=
W_{\{2,3\},\alpha_2+\alpha_3}
=\mathcal W.
\]

\noindent For $J=\{1,3\}$, the two $\mathfrak{sl}_2$-actions commute, and
\[
\begin{aligned}
W_{\{1,3\},0}
&=W_{\{1\},0}\cap W_{\{3\},0}\\
&=\operatorname{span}
  \{f_1,f_2-f_3-f_4+f_5,f_6\},\\
W_{\{1,3\},\alpha_1}
&=W_{\{3\},0},\\
W_{\{1,3\},\alpha_3}
&=W_{\{1\},0},\\
W_{\{1,3\},\alpha_1+\alpha_3}
&=\mathcal W.
\end{aligned}
\]
The highest weights of the $M_{\{1,3\}}$-summands
meeting $\mathcal W$ are
$0,\alpha_1,\alpha_3,\alpha_1+\alpha_3$. Finally, the decomposition into irreducible $\SL_4$-modules is
\[
\Lambda^2V_0\otimes\Lambda^2V_0
=
V(2\omega_2)\oplus
V(\omega_1+\omega_3)\oplus V(0).
\]
If we set
\[
g_0=f_1-f_2+f_3+f_4-f_5+f_6,
\]
the zero-weight parts of the corresponding filtration spaces
for $J=\{1,2,3\}$ are
\[
\begin{aligned}
W_{\{1,2,3\},0}
&=\operatorname{span}\{g_0\},\\
W_{\{1,2,3\},\omega_1+\omega_3}
&=\operatorname{span}
  \{f_1-f_6,f_2-f_5,f_3-f_4,g_0\},\\
W_{\{1,2,3\},2\omega_2}
&=\mathcal W.
\end{aligned}
\]

We now define
\[
\begin{aligned}
b_{12}&=f_1,\\
b_{13}&=f_2-f_1,\\
b_{14}&=f_3-f_2+f_1,\\
b_{23}&=f_4-f_2+f_1,\\
b_{24}&=f_5+f_2-f_3-f_4-2f_1,\\
b_{34}&=g_0.
\end{aligned}
\]
The tensor-product basis is expressed in terms of these vectors by
\[
\begin{aligned}
f_{12}&=b_{12},\\
f_{13}&=b_{12}+b_{13},\\
f_{14}&=b_{13}+b_{14},\\
f_{23}&=b_{13}+b_{23},\\
f_{24}&=b_{12}+b_{13}+b_{14}+b_{23}+b_{24},\\
f_{34}&=b_{12}+b_{24}+b_{34}.
\end{aligned}
\]
For two-element subsets of $[4]$, the order in
Theorem~\ref{thm:two-factor-uniqueness} is the componentwise order:
\[
\{a'_1<a'_2\}\preceq\{a_1<a_2\}
\iff
a'_1\leq a_1,\quad a'_2\leq a_2.
\]
Thus all the displayed coefficients are nonnegative integers,
the diagonal coefficients are $1$, and every $b_{A'}$
occurring in $f_A$ satisfies $X_{A'}\preceq X_A$.
This verifies condition~(i).

For condition~(ii), the highest weights $\eta$ for which
$V_{J,\eta}\cap\mathcal W\ne0$ are
\[
\begin{array}{c|l}
J&\eta\\ \hline
\varnothing&0\\
\{i\}&0,\alpha_i\\
\{1,2\}&0,\alpha_1+\alpha_2\\
\{2,3\}&0,\alpha_2+\alpha_3\\
\{1,3\}&0,\alpha_1,\alpha_3,\alpha_1+\alpha_3\\
\{1,2,3\}&0,\omega_1+\omega_3,2\omega_2.
\end{array}
\]
The following table gives the distinct nonzero filtration
spaces $W_{J,\mu}$, with one representative $\mu$ for each:
\[
\renewcommand{\arraystretch}{1.2}
\begin{array}{c|c|l}
J&\mu&W_{J,\mu}\\
\hline
\varnothing&0&\mathcal W\\
\hline
\{1\}&0&
\operatorname{span}\{b_{12},b_{23},b_{24},b_{34}\}\\
\{1\}&\alpha_1&\mathcal W\\
\{2\}&0&
\operatorname{span}\{b_{13},b_{14},b_{23},b_{34}\}\\
\{2\}&\alpha_2&\mathcal W\\
\{3\}&0&
\operatorname{span}\{b_{12},b_{14},b_{24},b_{34}\}\\
\{3\}&\alpha_3&\mathcal W\\
\hline
\{1,2\}&0&
\operatorname{span}\{b_{23},b_{34}\}\\
\{1,2\}&\alpha_1+\alpha_2&\mathcal W\\
\{2,3\}&0&
\operatorname{span}\{b_{14},b_{34}\}\\
\{2,3\}&\alpha_2+\alpha_3&\mathcal W\\
\hline
\{1,3\}&0&
\operatorname{span}\{b_{12},b_{24},b_{34}\}\\
\{1,3\}&\alpha_1&
\operatorname{span}\{b_{12},b_{14},b_{24},b_{34}\}\\
\{1,3\}&\alpha_3&
\operatorname{span}\{b_{12},b_{23},b_{24},b_{34}\}\\
\{1,3\}&\alpha_1+\alpha_3&\mathcal W\\
\hline
\{1,2,3\}&0&\operatorname{span}\{b_{34}\}\\
\{1,2,3\}&\omega_1+\omega_3&
\operatorname{span}\{b_{14},b_{23},b_{24},b_{34}\}\\
\{1,2,3\}&2\omega_2&\mathcal W
\end{array}
\]
Every space in the table is therefore spanned by a subset of
$\{b_{12},b_{13},b_{14},b_{23},b_{24},b_{34}\}$.
This verifies condition~(ii), and the uniqueness assertion in Theorem~\ref{thm:two-factor-uniqueness}
now gives
\[
b_A=e_{X_A},
\qquad
X_A=(A,A^c),
\qquad A\in\binom{[4]}2.
\]
Hence the MV basis of the six-dimensional zero-weight space is
\[
\{b_{12},b_{13},b_{14},b_{23},b_{24},b_{34}\}.
\]
Thus the unitriangularity condition and the filtration
conditions satisfying $(P1)$ determine all six MV basis vectors explicitly.

\end{example}

\subsection{Hodge duality and the symmetry theorem}
Let $V_0=\CC^r$ with ordered basis $e_1,\ldots,e_r$. For $A\subseteq[r]$, write
\[
A=\{a_1<\cdots<a_p\},
\qquad
e_A=e_{a_1}\wedge\cdots\wedge e_{a_p}.
\]
Define $\epsilon(A)\in\{\pm1\}$ by
\[
e_A\wedge e_{A^c}
=\epsilon(A)\,e_1\wedge\cdots\wedge e_r.
\]
The Hodge map
\begin{equation}
\label{eq:Hodge-map}
 D_p:\Lambda^pV_0\longrightarrow\Lambda^{r-p}V_0,
 \qquad D_p(e_A)=\epsilon(A)e_{A^c},
\end{equation}
satisfies
\begin{equation}
\label{eq:Hodge-intertwining}
 D_p\circ\Lambda^p(g)
 =\Lambda^{r-p}(g^{-T})\circ D_p
 \qquad(g\in\SL_r).
\end{equation}
For arbitrary $g\in\GL_r$ the identity is
\[
 D_p\Lambda^p(g)=\det(g)\,\Lambda^{r-p}(g^{-T})D_p,
\]
and restricting to $g\in\SL_r$ gives \eqref{eq:Hodge-intertwining}.

\begin{lemma}[Action of the Hodge map on the one-factor MV basis]
\label{lem:Hodge-MV-sign}
For every one-factor MV basis vector $v_A$ of $V(\omega_p)$ there is a
sign $s_A\in\{\pm1\}$ such that the fixed Hodge map $D_p$ of
\eqref{eq:Hodge-map} satisfies
\begin{equation}
\label{eq:Hodge-MV-sign}
 D_p(v_A)=s_Av^*_{A^c},
\end{equation}
where $v^*_{A^c}$ is the one-factor MV basis vector of
$V(\omega_{r-p})$.
\end{lemma}
\begin{proof}
Every weight of $V(\omega_p)$ is extremal, so by
\cite[Theorem~5.2 and Remark~2.10]{BaumannKamnitzerKnutson2021},
the normalized extremal vector at each weight is the corresponding
one-factor MV basis vector, up to sign. In the exterior-power model,
we choose the highest vectors
\[
 e_{[p]}=e_1\wedge\cdots\wedge e_p,
 \qquad
 e_{[p+1,r]}=e_{p+1}\wedge\cdots\wedge e_r.
\]
For $1\leq i<r$, let
\[
 n_i=I_{i-1}\oplus
 \begin{pmatrix}0&1\\-1&0\end{pmatrix}
 \oplus I_{r-i-1},
 \qquad n_i^{-T}=n_i.
\]
Given $A\in\binom{[r]}p$, choose $w\in\mathfrak S_r$ with
$w([p])=A$ and a product $n_w$ of the $n_i$. Then
\[
 v_A=\delta_A\,\Lambda^p(n_w)e_{[p]}=\pm e_A,
 \qquad
 v^*_{A^c}=\delta^*_{A^c}\,\Lambda^{r-p}(n_w)e_{[p+1,r]}
             =\pm e_{A^c},
\]
with $\delta_A,\delta^*_{A^c}\in\{\pm1\}$. Since
$D_p(e_{[p]})=e_{[p+1,r]}$ and
\eqref{eq:Hodge-intertwining} gives
\[
 D_p\Lambda^p(n_w)
 =\Lambda^{r-p}(n_w^{-T})D_p
 =\Lambda^{r-p}(n_w)D_p,
\]
we obtain
\[
 D_p(v_A)=\delta_A\Lambda^{r-p}(n_w)e_{[p+1,r]}
 =\delta_A(\delta^*_{A^c})^{-1}v^*_{A^c}.
\]
The scalar is $\pm1$, proving \eqref{eq:Hodge-MV-sign}.
\end{proof}

\begin{corollary}[Symmetry for two fundamental factors]
\label{cor:two-factor-symmetry}
For $G=\SL_r(\CC)$, $1\leq p,q\leq r-1$, and
$\boldsymbol\lambda=(\omega_p,\omega_q)$, Conjecture~\ref{conj:BGL}
holds on every weight space of $V(\omega_p)\otimes V(\omega_q)$.
\end{corollary}

\begin{proof}
Let
\[
 V=V(\omega_p)\otimes V(\omega_q),
 \qquad
 V^*=V(\omega_{r-q})\otimes V(\omega_{r-p}),
 \qquad
 \Theta(g)=g^{-T}.
\]
Define
\begin{equation}
\label{eq:T-Hodge}
 \mathcal H_{p,q}=(D_q\otimes D_p)\circ\operatorname{flip}:V\longrightarrow V^*.
\end{equation}
Equations~\eqref{eq:Hodge-intertwining} give
\begin{equation}
\label{eq:T-intertwining}
 \mathcal H_{p,q}\rho(g)=\rho^*(\Theta(g))\mathcal H_{p,q}
 \qquad(g\in\SL_r),
\end{equation}
represented by the commutative diagram
\[
\begin{CD}
 V @>{\mathcal H_{p,q}}>> V^*\\
 @V{\rho(g)}VV @VV{\rho^*(\Theta(g))}V\\
 V @>>{\mathcal H_{p,q}}> V^*.
\end{CD}
\]
For every Levi subgroup $M_J$, $\Theta(M_J)=M_J$ and
$\Theta(B_J)=B_J^-$. Thus an irreducible $M_J$-module of highest
weight $\eta$ is sent to one of highest weight
$\eta^\dagger=-w_0^J\eta$. Since
\[
 \mu-\eta\in Q_J^+
 \quad\Longrightarrow\quad
 \mu^\dagger-\eta^\dagger
 =-w_0^J(\mu-\eta)\in Q_J^+,
\]
we obtain
\begin{equation}
\label{eq:T-P1}
 \mathcal H_{p,q}(V_{J,\leq_J\mu}\cap V_\nu)
 =V^*_{J,\leq_J\mu^\dagger}\cap V^*_{-\nu}.
\end{equation}

\noindent For a tensor-product label $X=(A,B)$, set
$s_X=s_As_B$, where the signs are those of
Lemma~\ref{lem:Hodge-MV-sign}. Then
\begin{equation}
\label{eq:T-on-product-labels}
 \mathcal H_{p,q}(f_X)=s_Xf^*_{\sigma(X)},
 \qquad \sigma(A,B)=(B^c,A^c).
\end{equation}
By the factorization bijection of
\cite[Corollary~4.10]{BGL2024}, each label $Y$
corresponds to a pair of one-factor labels $(A_Y,B_Y)$.
If we set
\[
s_Y=s_{A_Y}s_{B_Y},
\]
then
\[
\widetilde e_{\sigma(Y)}=s_Y\mathcal H_{p,q}(e_Y).
\]
By \eqref{eq:T-P1}, the basis
$\widetilde{\mathcal E}^*=\{\widetilde e_{\sigma(Y)}\}$ satisfies
(P1) for every standard Levi subgroup. So, now it remains to check triangularity. If $a_{Y,X}\ne0$, Corollary~5.12
and the equality of the total weights of $Y$ and $X$ give
\[
 \eps_{A_Y}-\eps_{A_X}\in Q_+,
 \qquad
 \eps_{B_Y^c}-\eps_{B_X^c}
 =-\eps_{B_Y}+\eps_{B_X}
 =\eps_{A_Y}-\eps_{A_X}\in Q_+.
\]
Hence $Y\preceq X$, which implies $\sigma(Y)\preceq\sigma(X)$. Applying
$\mathcal H_{p,q}$ to \eqref{eq:MV-unitriangular-support} and using
\eqref{eq:T-on-product-labels}, we get
\begin{align}
 s_Xf^*_{\sigma(X)}
 &=\sum_Ya_{Y,X}\mathcal H_{p,q}(e_Y)
 =\sum_Ys_Ya_{Y,X}\widetilde e_{\sigma(Y)},\\
 f^*_{\sigma(X)}
 &=\sum_Ys_Xs_Ya_{Y,X}\widetilde e_{\sigma(Y)}.
 \label{eq:transferred-matrix}
\end{align}
The last matrix is unitriangular, since $s_X^2=1$, and has the
required support. The actual MV basis of $V^*$ also satisfies (P1)
and the same unitriangularity. Theorem~\ref{thm:two-factor-uniqueness} applied
to the weight space $V^*_{-\nu}$ therefore gives
\[
 \widetilde e_{\sigma(Y)}=e^*_{\sigma(Y)}.
\]
Consequently
\begin{equation}
\label{eq:coefficients-up-to-sign}
 a^*_{\sigma(Y),\sigma(X)}=s_Xs_Ya_{Y,X}.
\end{equation}
By \cite[Theorem~5.8]{BGL2024}, both transition matrices have nonnegative
integer entries. Thus, if $a_{Y,X}>0$, then
$s_Xs_Y=1$ in \eqref{eq:coefficients-up-to-sign}; if
$a_{Y,X}=0$, the same equation gives
$a^*_{\sigma(Y),\sigma(X)}=0$. Therefore
\[
 a^{(\omega_p,\omega_q)}_{Y,X}
 =a^{(\omega_{r-q},\omega_{r-p})}_{\sigma(Y),\sigma(X)}
\]
for every pair of labels of weight $\nu$. In addition,
\begin{equation}
\label{eq:two-factor-hodge-sign-compatibility}
 a^{(\omega_p,\omega_q)}_{Y,X}\ne0
 \quad\Longrightarrow\quad s_Xs_Y=1.
\end{equation}
This is Conjecture~\ref{conj:BGL} for $(\omega_p,\omega_q)$ together
with the sign-compatibility statement needed for the induction in
Section~\ref{sec:three-factors}.
\end{proof}

\section{Three factors: factorization and subsequence inequalities}
\label{sec:three-factors}

\subsection{Factorization at a vanishing partial sum}
For a sequence $\boldsymbol\lambda=(\lambda_1,\ldots,\lambda_n)$ and
labels $X=(X_1,\ldots,X_n)$, $Y=(Y_1,\ldots,Y_n)$ of the same total
weight, let
\begin{equation}
\label{eq:three-factor-delta}
 \Delta_k(Y,X)=\sum_{j=1}^k
       \bigl(\wt(Y_j)-\wt(X_j)\bigr),
 \qquad 1\leq k<n.
\end{equation}
We call $(Y,X)$ \emph{admissible} if
\begin{equation}
\label{eq:three-factor-admissible}
 \Delta_k(Y,X)\in Q_+\qquad(1\leq k<n).
\end{equation}
A necessary condition for $a^{\boldsymbol\lambda}_{Y,X}\ne0$ is
\eqref{eq:three-factor-admissible}
\cite[Corollary~5.12]{BGL2024}.

\begin{lemma}[Support and normalization {\cite[Theorem~5.13]{BGL2024}}]
\label{fact:two-block-b}
Let $\boldsymbol\lambda=(\boldsymbol\lambda^{(1)},
\boldsymbol\lambda^{(2)})$ be a factorization into two consecutive blocks and 
$Z=(Z^{(1)},Z^{(2)}).$ If
\[
 b_{Z',(W_1,W_2)}\ne0,
\]
then either $Z'=(W_1,W_2)$ or
$Z'^{(1)}\subsetneq W_1$ as cycles in the first convolution
Grassmannian. Moreover,
\[
 b_{(W_1,W_2),(W_1,W_2)}=1.
\]
\end{lemma}

\begin{lemma}[Dimension of a convolution MV cycle]
\label{fact:mv-dimension}
For a cycle $Z\in\mathcal Z(\boldsymbol\lambda)_\mu$,
\[
 \dim Z=\langle\rho,|\boldsymbol\lambda|-\mu\rangle.
\]
In particular, two cycles in the same fiber
$\mathcal Z(\boldsymbol\lambda)_\mu$ have equal dimension, \cite[\S~2.3]{BGL2024}.
\end{lemma}

\begin{lemma}[Factorization at a vanishing partial sum]
\label{lem:prefix-factorization-final}
Let $\boldsymbol\lambda=(\lambda_1,\ldots,\lambda_n)$ be any sequence of
dominant weights, and let $1\leq k<n$. If
$\Delta_k(Y,X)=0$, then
\begin{equation}
\label{eq:prefix-factorization-final}
 a^{\boldsymbol\lambda}_{Y,X}
 =a^{\boldsymbol\lambda_{\leq k}}_{Y_{\leq k},X_{\leq k}}
  a^{\boldsymbol\lambda_{>k}}_{Y_{>k},X_{>k}}.
\end{equation}
\end{lemma}
\begin{proof}
From \cite[\S~5.6, Proposition~5.10]{BGL2024} we get
\begin{equation}
\label{eq:prefix-associativity-sum}
 a^{\boldsymbol\lambda}_{Y,X}
 =\sum_{W_1,W_2}b_{Y,(W_1,W_2)}
 a^{\boldsymbol\lambda_{\leq k}}_{W_1,X_{\leq k}}
 a^{\boldsymbol\lambda_{>k}}_{W_2,X_{>k}}.
\end{equation}
For a nonzero summand, Lemma~\ref{fact:two-block-b} gives either
$Y=(W_1,W_2)$ or $Y_{(1)}\subsetneq W_1$. Here $Y_{(1)}$ denotes
the first convolution cycle associated with
$Y_{\leq k}$, and $W_1$ is the first convolution cycle in the
intermediate pair $(W_1,W_2)$. Since $\Delta_k(Y,X)=0$, we have
$\wt(Y_{\leq k})=\wt(X_{\leq k})=:\mu_{\leq k}$; moreover,
the nonvanishing of $a_{W_1,X_{\leq k}}$ gives
$\wt(W_1)=\wt(X_{\leq k})$. Thus both $Y_{(1)}$ and $W_1$ are
irreducible components in the same convolution fiber
$\mathcal Z(\boldsymbol\lambda_{\leq k})_{\mu_{\leq k}}$, and
\[
 \wt(Y_{(1)})=\wt(Y_{\leq k})
 =\wt(X_{\leq k})=\wt(W_1).
\]
Lemma~\ref{fact:mv-dimension} therefore gives
\[
 \dim Y_{(1)}=\dim W_1.
\]
Since a strict inclusion of irreducible cycles of equal dimension is
impossible, $Y_{(1)}=W_1$, so Lemma~\ref{fact:two-block-b} forces $Y=(W_1,W_2)$ and
$b_{Y,(W_1,W_2)}=1$.
Thus exactly one summand in \eqref{eq:prefix-associativity-sum} can be
nonzero, and it is the product on the right-hand side of
\eqref{eq:prefix-factorization-final}.
\end{proof}

\subsection{A conditional symmetry theorem}
We use the label identification
$\mathcal Z(\boldsymbol\lambda)\simeq
\mathcal Z(\lambda_1)\otimes\cdots\otimes\mathcal Z(\lambda_n)$
\cite[Corollary~4.10]{BGL2024}. For such a label $X$, let
$\eta_J(X)$ denote the highest weight of its connected $J$-crystal
component. For an admissible pair $(Y,X)$, we will consider the condition
\begin{equation}
\label{eq:Levi-highest-weight-order}
 \eta_J(Y)\leq_J\eta_J(X)
 \qquad\text{for every standard Levi }M_J.
\end{equation}

\noindent For each one-factor extremal vector of the MV basis, let $s_A\in\{\pm1\}$ be the
sign defined by the one-factor Hodge map in
Lemma~\ref{lem:Hodge-MV-sign}. If $X=(A_1,\ldots,A_n)$ is a
tensor-product label, let
\[
 s_X=\prod_{j=1}^n s_{A_j}.
\]
If $Z\in\mathcal Z(\boldsymbol\lambda)$,  we write
$(A_1(Z),\ldots,A_n(Z))$ for its image under this identification
\cite[Corollary~4.10]{BGL2024} and set
\[
 s_Z=\prod_{j=1}^n s_{A_j(Z)}.
\]
The tensor-product rule gives
$s_Z=s_{Z_{\leq k}}s_{Z_{>k}}$ for every cut $k$.
We say that a sequence has \emph{Hodge-sign compatibility} if
\begin{equation}
\label{eq:HSC}
 a^{\boldsymbol\lambda}_{Z,X}\ne0
 \quad\Longrightarrow\quad s_Zs_X=1.
\end{equation}
For one factor this is immediate; for two factors it is
\eqref{eq:two-factor-hodge-sign-compatibility}.

\begin{theorem}[Conditional symmetry theorem]
\label{thm:conditional-symmetry}
Let $G=\SL_r$ and let
$\boldsymbol\lambda=(\omega_{p_1},\ldots,\omega_{p_n})$ be a sequence of
minuscule weights, with $n\geq3$. Assume:
\begin{enumerate}[label=\textup{(\roman*)}]
 \item Symmetry and Hodge-sign compatibility hold for every proper
 contiguous subsequence of $\boldsymbol\lambda$;
 \item every admissible pair $(Y,X)$ with $X\neq Y$ satisfying
 \eqref{eq:Levi-highest-weight-order} has
 $\Delta_k(Y,X)=0$ for some $1\leq k<n$.
\end{enumerate}
Then Conjecture \ref{conj:BGL} and Hodge-sign compatibility hold for
$\boldsymbol\lambda$.
\end{theorem}
\begin{proof}
We write
\[
 \mathbf F_X=\sum_Z a_{Z,X}\mathbf E_Z,
 \qquad \mathbf M=(a_{Z,X}),
\]
where $\mathbf F_X$ is the tensor product MV basis and $\mathbf E_Z$ the MV basis of the tensor product. Let
\[
 \mathcal H_{\boldsymbol\lambda}
 =(D_{p_n}\otimes\cdots\otimes D_{p_1})\circ\operatorname{rev},
 \qquad
 \boldsymbol\lambda^*=(\lambda_n^*,\ldots,\lambda_1^*).
\]
The one-factor sign definition gives
\[
 \mathcal H_{\boldsymbol\lambda}(\mathbf F_X)
 =s_X\mathbf F^*_{\sigma X}.
\]
Define
\[
 \widetilde{\mathbf E}_{\sigma Z}
 =s_Z\mathcal H_{\boldsymbol\lambda}(\mathbf E_Z).
\]
Then
\begin{equation}
\label{eq:three-factor-A-prime}
 \mathbf F^*_{\sigma X}
 =\sum_Z s_Xs_Za_{Z,X}\widetilde{\mathbf E}_{\sigma Z},
 \qquad
 \mathbf M'_{\sigma Z,\sigma X}=s_Xs_Za_{Z,X}.
\end{equation}
By the intertwining identity for $D_p$,
\[
 \mathcal H_{\boldsymbol\lambda}\rho(g)
 =\rho^*(g^{-T})\mathcal H_{\boldsymbol\lambda},
 \qquad g\in\SL_r,
\]
and $g\mapsto g^{-T}$ preserves every Levi subgroup. If
$\mu^\dagger=-w_0^J\mu$, then
\[
 \mu-\eta\in Q_J^+
 \quad\Longrightarrow\quad
 \mu^\dagger-\eta^\dagger
 =-w_0^J(\mu-\eta)\in Q_J^+.
\]
Thus $\mathcal H_{\boldsymbol\lambda}$ transports every filtration satisfying $P1$ to the corresponding filtration, and $\eta_J(\sigma Z)=(-w_0^J)\eta_J(Z)$. Corollary~5.12 and
\[
 \Delta_{n-k}(\sigma Y,\sigma X)=\Delta_k(Y,X)
\]
show that $\mathbf M'$ is unitriangular for the same admissibility order as the
target transition matrix $\mathbf M^*$. We reindex the rows and columns of the target matrices by $\sigma$, so
$\mathbf M^*_{Y,X}=a^*_{\sigma Y,\sigma X}$ and we write
\[
 \widetilde{\mathbf E}=\mathbf E^*\mathbf U,
 \qquad \mathbf M^*=\mathbf U\mathbf M'.
\]
Both $\mathbf M'$ and $\mathbf M^*$ are unitriangular and supported on
the partial order
\[
 Y\preceq X\quad\Longleftrightarrow\quad
 \Delta_k(Y,X)\in Q_+\quad(1\leq k<n).
\]
Hence $\mathbf U$ is unitriangular and has the same support property. If there is a Levi subgroup $M_J$ for which
$\eta_J(Y)\not\leq_J\eta_J(X)$, put
$\mu=\eta_J(X)$ and $\mu^\dagger=-w_0^J\mu$, and consider
\[
 S=V^*_{J,\leq_J\mu^\dagger}.
\]
Since $-w_0^J$ maps $Q_J^+$ onto itself and
$\eta_J(\sigma Z)=(-w_0^J)\eta_J(Z)$, we have
\[
 \sigma Z\in S
 \quad\Longleftrightarrow\quad
 \eta_J(Z)\leq_J\mu.
\]
Thus, before reindexing, $S$ contains the basis vector indexed by
$\sigma X$ and not that indexed by $\sigma Y$. After reindexing
by $\sigma$, the common index subset for $S$ contains $X$ and not
$Y$. Both $\mathbf E^*$ and $\widetilde{\mathbf E}$ are bases of
$V^*$ satisfying $P1$ for this same subspace, and they have the same
unitriangular flag $F_0\subset F_1\subset\cdots\subset F_N$.
The index subset spanning $S$ is therefore determined by
\[
 {\bf1}_{\{e_i\in S\}}=
 \dim(S\cap F_i)-\dim(S\cap F_{i-1}).
\]
Thus both bases use the same subset of indices, and therefore
\begin{equation}
\label{eq:U-Levi-distinguished}
 \mathbf U_{Y,X}=0\qquad\text{whenever some Levi subgroup distinguishes $Y$ from $X$.}
\end{equation}

\noindent Suppose next that $(Y,X)$ is admissible and satisfies
\eqref{eq:Levi-highest-weight-order}. By (ii),
$\Delta_k(Y,X)=0$ for some $k$. Lemma~\ref{lem:prefix-factorization-final}
and the induction hypothesis give
\begin{align}
 a^{\boldsymbol\lambda}_{Y,X}
 &=a^{\boldsymbol\lambda_{\leq k}}_{Y_{\leq k},X_{\leq k}}
   a^{\boldsymbol\lambda_{>k}}_{Y_{>k},X_{>k}},
 \label{eq:prefix-factor-source}\\
 a^{\boldsymbol\lambda^*}_{\sigma Y,\sigma X}
 &=a^{(\boldsymbol\lambda_{>k})^*}_{\sigma Y_{>k},\sigma X_{>k}}
   a^{(\boldsymbol\lambda_{\leq k})^*}_{\sigma Y_{\leq k},\sigma X_{\leq k}}
 =a^{\boldsymbol\lambda}_{Y,X}.
 \label{eq:prefix-factor-target}
\end{align}
If $a^{\boldsymbol\lambda}_{Y,X}\ne0$, both shorter coefficients in
\eqref{eq:prefix-factor-source} are nonzero. Their Hodge-sign
compatibility gives
\[
 s_Xs_Y
 =\bigl(s_{X_{\leq k}}s_{Y_{\leq k}}\bigr)
  \bigl(s_{X_{>k}}s_{Y_{>k}}\bigr)=1.
\]
If $a^{\boldsymbol\lambda}_{Y,X}=0$, then both products in
\eqref{eq:prefix-factor-source}--\eqref{eq:prefix-factor-target}
vanish. In either case,
\begin{equation}
\label{eq:Astar-Aprime-Levi-ordered}
 \mathbf M^*_{Y,X}=\mathbf M'_{Y,X}.
\end{equation}

\noindent So, from $\mathbf M^*=\mathbf U\mathbf M'$ we obtain
\begin{equation}
\label{eq:U-recursion-final}
 \mathbf M^*_{Y,X}-\mathbf M'_{Y,X}
 =\mathbf U_{Y,X}+\sum_{Y\prec W\prec X}\mathbf U_{Y,W}\mathbf M'_{W,X}.
\end{equation}
Let $Z_1,\ldots,Z_N$ be a linear extension of $\preceq$, and
write $i(Z_j)=j$. We prove by induction on $j-i$ that
\[
Y=Z_i\prec X=Z_j
\quad\Longrightarrow\quad
\mathbf U_{Y,X}=0.
\]
For such a pair, either
\[
\exists\,J:\quad
\eta_J(Y)\not\leq_J\eta_J(X)
\quad\overset{\eqref{eq:U-Levi-distinguished}}{\Longrightarrow}\quad
\mathbf U_{Y,X}=0,
\]
or
\[
\begin{gathered}
\eta_J(Y)\leq_J\eta_J(X)\quad\forall J
\\[2pt]
\overset{\eqref{eq:Levi-highest-weight-order},\,\textup{(ii)}}{\Longrightarrow}
\exists\,k:\quad\Delta_k(Y,X)=0
\\[2pt]
\overset{\eqref{eq:Astar-Aprime-Levi-ordered}}{\Longrightarrow}
\mathbf M^*_{Y,X}-\mathbf M'_{Y,X}=0.
\end{gathered}
\]
Moreover,
\[
\sum_{Y\prec W\prec X}
\mathbf U_{Y,W}\mathbf M'_{W,X}=0,
\]
since the sum is empty for $j-i=1$, while for $j-i>1$,
\[
Y\prec W\prec X
\quad\Longrightarrow\quad
i(W)-i(Y)<j-i
\quad\overset{\mathrm{induction}}{\Longrightarrow}\quad
\mathbf U_{Y,W}=0.
\]
Consequently, \eqref{eq:U-recursion-final} yields
\[
0
=\mathbf M^*_{Y,X}-\mathbf M'_{Y,X}
=\mathbf U_{Y,X}
+\sum_{Y\prec W\prec X}
\mathbf U_{Y,W}\mathbf M'_{W,X}
=\mathbf U_{Y,X}.
\]
Thus 
\[
 \mathbf U=\mathrm{Id},\qquad \mathbf M^*=\mathbf M'.
\]
The coefficients satisfy $a_{Y,X}\geq0$ and $a^*_{\sigma Y,\sigma X}\geq0$
\cite[Theorem~5.8]{BGL2024}.
From \eqref{eq:three-factor-A-prime},
\[
 a^*_{\sigma Y,\sigma X}=s_Xs_Ya_{Y,X}.
\]
Hence $a_{Y,X}>0$ implies $s_Xs_Y=1$, and $a_{Y,X}=0$ implies
$a^*_{\sigma Y,\sigma X}=0$. This proves both the symmetry conjecture and
\eqref{eq:HSC}.
\end{proof}

\subsection{The subsequence criterion for $1\leq s\leq2$}
For
\[
 \boldsymbol\lambda=(\omega_1,\omega_s,\omega_q),
 \qquad 1\leq s\leq\min\{2,r-1\},\quad 1\leq q\leq r-1,
\]
we write $X=(A,B,C)$ and $Y=(A',B',C')$ for the corresponding
three-subset labels. In this family, condition \textup{(H)} is
\begin{equation}
\label{eq:H-three-factor-family}
\begin{gathered}
 (Y,X)\text{ admissible},\quad Y\ne X,\quad
 \eta_J(Y)\leq_J\eta_J(X)\quad\text{for every}M_J
 \\
 \Longrightarrow\qquad
 \Delta_1(Y,X)=0\quad\text{or}\quad\Delta_2(Y,X)=0.
\end{gathered}
\end{equation}

\begin{theorem}[Levi-filtration criterion for $(\omega_1,\omega_s,\omega_q)$]
\label{thm:three-factor-H}
For every $r\geq2$, $1\leq s\leq\min\{2,r-1\}$ and
$1\leq q\leq r-1$, condition~\eqref{eq:H-three-factor-family} holds.
\end{theorem}
\begin{proof}
Let $(Y,X)$ be admissible, $Y\ne X$, and suppose
\[
 \Delta_1(Y,X)\ne0,\qquad \Delta_2(Y,X)\ne0.
\]
For $0\leq t\leq r$, define
\begin{align*}
 \Delta_1^{(t)}(Y,X)
 &=\sum_{i=1}^t\bigl(\mathbf1_{i\in A'}-\mathbf1_{i\in A}\bigr),\\
 \Delta_2^{(t)}(Y,X)
 &=\sum_{i=1}^t\bigl(\mathbf1_{i\in A'}+\mathbf1_{i\in B'}
                -\mathbf1_{i\in A}-\mathbf1_{i\in B}\bigr).
\end{align*}
By \eqref{eq:three-factor-admissible},
\[
\Delta_1(Y,X),\ \Delta_2(Y,X)\in Q_+.
\]
For $\gamma=\sum_{i=1}^r\gamma_i\eps_i$ with
$\sum_{i=1}^r\gamma_i=0$, we have
\[
\gamma
=\sum_{t=1}^{r-1}
\left(\sum_{i=1}^t\gamma_i\right)\alpha_t.
\]
Consequently,
\[
\begin{aligned}
\Delta_1(Y,X)
 &=\sum_{t=1}^{r-1}\Delta_1^{(t)}(Y,X)\alpha_t,\\
\Delta_2(Y,X)
 &=\sum_{t=1}^{r-1}\Delta_2^{(t)}(Y,X)\alpha_t.
\end{aligned}
\]
Thus
\[
\Delta_1^{(t)}(Y,X),\ \Delta_2^{(t)}(Y,X)\geq0
\qquad(1\leq t<r),
\]
and
\[
\Delta_k(Y,X)=0
\quad\Longleftrightarrow\quad
\Delta_k^{(t)}(Y,X)=0
\quad\text{for every }1\leq t<r.
\]
By Lemma~\ref{lem:mixed-subsequence-comparison}, there is an interval
$I=[u,v]\subseteq[r]$, with $u<v$, such that
\begin{equation}
\label{eq:subsequence-comparison}
 \operatorname{dec}_Y(I)<\operatorname{dec}_X(I)
 \qquad\text{or}\qquad
 \operatorname{inc}_Y(I)>\operatorname{inc}_X(I).
\end{equation}
Let $J=\{\alpha_u,\ldots,\alpha_{v-1}\}$, and 
\[
 w_Z(I)=A_Z(I)^\downarrow B_Z(I)^\downarrow C_Z(I)^\downarrow,
 \qquad \lambda_Z=\operatorname{sh}(w_Z(I)),\quad Z\in\{X,Y\},
\]
where $\operatorname{sh}(w)$ denotes the shape of the Schensted
row-insertion tableau of $w$. By Schensted's theorem,
\[
\begin{aligned}
\operatorname{sh}(w)_1
 &=\max\{|S|:S\text{ is a weakly increasing subsequence of }w\},\\
\ell(\operatorname{sh}(w))
 &=\max\{|S|:S\text{ is a strictly decreasing subsequence of }w\}.
\end{aligned}
\]
\noindent We now make the word-order convention explicit. For a single
type-$A$ minuscule factor, the corresponding tableau is a column, and its
column-reading word (read from bottom to top) is the decreasing word
$A_Z(I)^\downarrow$. Write $T_S$ for the column tableau whose entries
are the elements of $S$. Under the convention of Section~\ref{sec:prelim},
the label $Z=(A_Z,B_Z,C_Z)$ corresponds to the ordinary Kashiwara tensor
$T_{C_Z(I)}\otimes T_{B_Z(I)}\otimes T_{A_Z(I)}$. The standard
word realization of a tensor product reads these factor words in reverse
order; hence the plactic word is
\[
 w_Z(I)=A_Z(I)^\downarrow B_Z(I)^\downarrow C_Z(I)^\downarrow.
\]
By the type-$A$ plactic/crystal correspondence, its row-insertion shape
is the highest weight of the connected component containing $Z$. Then
\[
 \eta_J(Z)|_I=\operatorname{sh}(w_Z(I)),
 \qquad
 \eta_J(Z)|_{[r]\setminus I}=\nu|_{[r]\setminus I},
 \quad Z\in\{X,Y\},
\]
where $\nu=\operatorname{wt}(X)=\operatorname{wt}(Y)$, \cite[Corollary~4.10 and Remark~4.11]{BGL2024} together with the
plactic description of type $A$ crystals
\cite{KashiwaraNakashima1994,HongKang2002}. For example, for $r=3$,
$I=[1,3]$, and $Z=(\{1\},\{2\},\{3\})$, then
\[
\begin{aligned}
 A_Z(I)^\downarrow B_Z(I)^\downarrow C_Z(I)^\downarrow&=123,
 &\operatorname{sh}(123)&=(3),\\
 C_Z(I)^\downarrow B_Z(I)^\downarrow A_Z(I)^\downarrow&=321,
 &\operatorname{sh}(321)&=(1,1,1).
\end{aligned}
\]
Thus reversing the block order changes the row-insertion shape.
\medskip

\noindent So, by the previous fact,
\begin{equation}
\label{eq:eta-is-RSK-shape}
 \eta_J(Z)|_I=\lambda_Z,
 \qquad
 \eta_J(Z)|_{[r]\setminus I}=\nu|_{[r]\setminus I},
 \quad Z\in\{X,Y\}.
\end{equation}
Each of the three blocks of $w_Z(I)$ is strictly decreasing, hence a weakly
increasing subsequence uses at most one letter from each block, so its
length is at most $3$. Schensted's theorem \cite{Schensted1961} and
Greene's theorem \cite{Greene1974} give
\begin{equation}
\label{eq:Greene-RSK-invariants}
 \operatorname{dec}_Z(I)=\ell(\lambda_Z),
 \qquad
 \operatorname{inc}_Z(I)=\mathbf1_{\lambda_{Z,1}=3}.
\end{equation}
The two words have the same content by the profile identity and in particular
$|\lambda_X|=|\lambda_Y|$. If
$\eta_J(Y)\leq_J\eta_J(X)$, then, by
\eqref{eq:eta-is-RSK-shape},
\begin{equation}
\label{eq:partition-dominance}
 \lambda_{X,1}\geq\lambda_{Y,1},\qquad
 \sum_{i=1}^m\lambda_{X,i}\geq
 \sum_{i=1}^m\lambda_{Y,i}\quad(m\geq1).
\end{equation}
Since $|\lambda_X|=|\lambda_Y|$, the inequalities
\eqref{eq:partition-dominance} imply
\[
 \ell(\lambda_Y)\geq\ell(\lambda_X),
 \qquad
 \mathbf1_{\lambda_{Y,1}=3}\leq
 \mathbf1_{\lambda_{X,1}=3}.
\]
By \eqref{eq:Greene-RSK-invariants},
$\operatorname{dec}_Y(I)\geq\operatorname{dec}_X(I)$ and
$\operatorname{inc}_Y(I)\leq\operatorname{inc}_X(I)$, contradicting
\eqref{eq:subsequence-comparison}. Thus some standard Levi separates
$Y$ from $X$, proving~\eqref{eq:H-three-factor-family}.
\end{proof}

\begin{theorem}[Symmetry for $(\omega_1,\omega_s,\omega_q)$]
\label{thm:three-factor-BGL}
Let $G=\SL_r(\CC)$, $r\geq2$, and
$1\leq s\leq\min\{2,r-1\}$, $1\leq q\leq r-1$. Then
Conjecture~\ref{conj:BGL} holds on every weight space of
\[
 V(\omega_1)\otimes V(\omega_s)\otimes V(\omega_q).
\]
\end{theorem}
\begin{proof}
Every proper contiguous subsequence of
$(\omega_1,\omega_s,\omega_q)$ has at most two factors. The one-factor
case is immediate, and the two-factor symmetry and sign compatibility
are Corollary~\ref{cor:two-factor-symmetry} and
\eqref{eq:two-factor-hodge-sign-compatibility}. Theorem
\ref{thm:three-factor-H} supplies hypothesis~\textup{(ii)} of
Theorem~\ref{thm:conditional-symmetry}. Therefore
\[
 \mathbf M^*=\mathbf M',\qquad
 a^{(\omega_1,\omega_s,\omega_q)}_{Y,X}
 =a^{(\omega_{r-q},\omega_{r-s},\omega_{r-1})}_{\sigma(Y),\sigma(X)}
\]
for every pair of labels $X,Y$ of the same weight.
\end{proof}

\begin{remark}
\label{rem:scope}
Condition~\textup{(H)} does not hold for arbitrary triples. For
$G=\SL_4$, $\boldsymbol\lambda=(\omega_1,\omega_3,\omega_2)$, the labels
$Y=(1;123;34)$ and $X=(2;134;13)$ have the same total weight,
$\Delta_1(Y,X)=\alpha_1$,
$\Delta_2(Y,X)=\alpha_1+\alpha_2+\alpha_3$, and
$\eta_J(Y)\leq_J\eta_J(X)$ for every $M_J$. This example
shows that Theorem~\ref{thm:three-factor-H} cannot be extended to arbitrary
triples and that the present criterion leaves pairs of this kind unresolved.
\end{remark}

\subsection{Descent along Cartan projections}
\label{subsec:descent}
In this subsection $G$ is as in Section~\ref{sec:prelim} and
$\boldsymbol\mu$ is a sequence of dominant weights cut into consecutive
blocks $\boldsymbol\mu=(\boldsymbol\mu^{(1)},\ldots,\boldsymbol\mu^{(m)})$.
Let $\lambda_j=|\boldsymbol\mu^{(j)}|$ be the sum of the weights in the
$j$-th block and $\boldsymbol\lambda=(\lambda_1,\ldots,\lambda_m)$. For a
label $Z\in\Zc(\boldsymbol\mu)$ we write $Z=(Z_{(1)},\ldots,Z_{(m)})$ with
$Z_{(j)}\in\Zc(\boldsymbol\mu^{(j)})$.

\begin{remark}
\label{rem:BD-scope}
The main theorem of \cite{BaumannDemarais2021} (identification of the MV and
dual canonical bases) is specific to $\SL_2$. However, its Section~3, which
contains the statements we use, is written for an arbitrary connected
reductive group. Proposition~5 there (the bijection $Z\mapsto|Z|$ and the
formula $p(\langle Z\rangle)=\langle|Z|\rangle$) is a direct consequence of
\cite[Theorem~3.4]{BGL2024}, and Proposition~6 (with formula~(9) in its proof)
uses only locality of intersection multiplicities.
\end{remark}

\begin{proposition}[Descent of the symmetry]
\label{prop:descent}
If Conjecture~\ref{conj:BGL} holds for $\boldsymbol\mu$ and for each block
$\boldsymbol\mu^{(j)}$, then it holds for $\boldsymbol\lambda$.
\end{proposition}
\begin{proof}
Fix a weight. Let $\mathbf A=(a^{\boldsymbol\mu}_{Z,X})$,
$\mathbf D_{W,X}=\prod_{j=1}^m a^{\boldsymbol\mu^{(j)}}_{W_{(j)},X_{(j)}}$,
and define $\mathbf B=(b_{Z,W})$ by
\[
 \langle W_{(1)}\rangle\otimes\cdots\otimes\langle W_{(m)}\rangle
 =\sum_{Z}b_{Z,W}\,\langle Z\rangle,
\]
where $\langle W_{(j)}\rangle$ is the MV basis element of
$V(\boldsymbol\mu^{(j)})$.

\smallskip\noindent The first step is to expand each factor:
\[
 \langle\!\langle X\rangle\!\rangle
 =\bigotimes_{j=1}^m\ \sum_{W_{(j)}}
   a^{\boldsymbol\mu^{(j)}}_{W_{(j)},X_{(j)}}\langle W_{(j)}\rangle
 =\sum_W\mathbf D_{W,X}\sum_Z b_{Z,W}\langle Z\rangle,
 \qquad\text{hence}\qquad
 \mathbf A=\mathbf B\mathbf D.
\]
The matrix $\mathbf D$ is unitriangular, hence invertible. The same holds
for the dual sequence, with matrices $\mathbf A^*=\mathbf B^*\mathbf D^*$,
and $\sigma(Z)=(\sigma Z_{(m)},\ldots,\sigma Z_{(1)})$.

\smallskip\noindent The next step is to show the symmetry of $\mathbf B$. By hypothesis, $\mathbf A^*_{\sigma Z,\sigma X}=\mathbf A_{Z,X}$ and, block by
block, $\mathbf D^*_{\sigma W,\sigma X}=\mathbf D_{W,X}$. Therefore
\[
 \sum_W\mathbf B^*_{\sigma Z,\sigma W}\mathbf D_{W,X}
 =\mathbf A^*_{\sigma Z,\sigma X}
 =\mathbf A_{Z,X}
 =\sum_W\mathbf B_{Z,W}\mathbf D_{W,X},
\]
and invertibility of $\mathbf D$ gives
\begin{equation}
\label{eq:descent-B}
 \mathbf B^*_{\sigma Z,\sigma W}=\mathbf B_{Z,W}.
\end{equation}

\smallskip\noindent We continue with the Cartan projections. Following \cite[Sec. 3.4]{BaumannDemarais2021}, for a sequence
of dominant weights $\boldsymbol\nu=(\nu_1,\ldots,\nu_t)$
and $Z\in\Zc(\boldsymbol\nu)_\eta$, we define
\[
|Z|=
\begin{cases}
\overline{m_t(Z)\cap\Gr^{|\boldsymbol\nu|}},
&\text{if }m_t(Z)\cap\Gr^{|\boldsymbol\nu|}\neq\varnothing,\\[2mm]
\varnothing,
&\text{otherwise},
\end{cases}
\]
where the closure is taken inside $\mathscr T_\eta$. Let $C_j\subset\Zc(\boldsymbol\mu^{(j)})$ be the set of labels with
$|Z_{(j)}|\neq\varnothing$, $C=C_1\times\cdots\times C_m$, and
$\|Z\|=(|Z_{(1)}|,\ldots,|Z_{(m)}|)$ for $Z\in C$. By
\cite[Propositions~5 and~6 and formula~(9) of the proof]{BaumannDemarais2021}
and Remark~\ref{rem:BD-scope},
\begin{equation}
\label{eq:BD-b-equals-a}
 \|\cdot\|:C\to\Zc(\boldsymbol\lambda)\ \text{is a bijection},
 \qquad
 b_{Z,W}=a^{\boldsymbol\lambda}_{\|Z\|,\|W\|}\quad(Z,W\in C),
\end{equation}
and each map $Z_{(j)}\mapsto|Z_{(j)}|$ is an isomorphism of crystals from
$C_j$ onto $\Zc(\lambda_j)$. By
\cite[Corollary~4.10 and the sentence following it]{BGL2024}, the crystal on
$\mathcal Z(\boldsymbol\nu^{(j)})$ is identified, with the opposite
tensor-product convention, with
\[
B(\nu_{t_j}^{(j)})\otimes\cdots\otimes B(\nu_1^{(j)}).
\]
Under this identification, the Cartan summand
$V(\lambda_j)$, where
$\lambda_j=\nu_1^{(j)}+\cdots+\nu_{t_j}^{(j)}$,
corresponds to the connected component containing the
tensor product of the highest-weight elements.
By \cite[Proposition~5]{BaumannDemarais2021}, the labels $Z$ for which
$|Z|\ne\varnothing$ are precisely those whose corresponding
basis vectors have nonzero projection onto this Cartan
summand. Consequently, $C_j$ is exactly the connected
crystal component corresponding to $V(\lambda_j)$.

\smallskip\noindent Next, we must show that ($\sigma$ commutes with $\|\cdot\|$). The component $C_j$ contains an element of weight $w_0\lambda_j$, whose image
under $\sigma$ has weight $-w_0\lambda_j=\lambda_j^*$. This is the highest
weight of the dual block, with multiplicity one, so $\sigma(C_j)$ is the
Cartan component $C_j^*$ of the dual block. The two maps
\[
 \|\cdot\|\circ\sigma,\ \ \sigma\circ\|\cdot\|\ :\ C_j\longrightarrow\Zc(\lambda_j^*)
\]
interchange $\widetilde e_i$ and $\widetilde f_i$ and send a weight $\mu$ to $-\mu$. Hence $(\sigma\circ\|\cdot\|)^{-1}\circ(\|\cdot\|\circ\sigma)$ is a crystal
automorphism of the connected highest-weight crystal $C_j$, which is the
identity. Thus
\[
 \|\sigma Z\|=\sigma\|Z\|,\qquad \sigma(C)=C^*.
\]
\noindent So, for $Z,W\in C$,
\[
\begin{aligned}
a^{\boldsymbol\lambda^*}_{\sigma\|Z\|,\sigma\|W\|}
&=a^{\boldsymbol\lambda^*}_{\|\sigma Z\|,\|\sigma W\|}
&&\bigl(\|\sigma Z\|=\sigma\|Z\|,\ 
        \|\sigma W\|=\sigma\|W\|\bigr)\\
&\overset{\eqref{eq:BD-b-equals-a}}{=}
\mathbf B^*_{\sigma Z,\sigma W}\\
&\overset{\eqref{eq:descent-B}}{=}
\mathbf B_{Z,W}\\
&\overset{\eqref{eq:BD-b-equals-a}}{=}
a^{\boldsymbol\lambda}_{\|Z\|,\|W\|},
\end{aligned}
\] 
and since $\|\cdot\|:C\to\Zc(\boldsymbol\lambda)$ is onto, every pair of labels of
the same weight is of the form $(\|Z\|,\|W\|)$, which proves the claim.
\end{proof}

\begin{theorem}[Symmetry with a non-fundamental factor]
\label{thm:nonfundamental}
Let $G=\SL_r(\CC)$, $r\geq2$, $1\leq s\leq\min\{2,r-1\}$ and
$1\leq q\leq r-1$. Then Conjecture~\ref{conj:BGL} holds on every weight space
for each of the sequences
\[
 (\omega_1+\omega_s,\ \omega_q),\qquad
 (\omega_1,\ \omega_s+\omega_q),
\]
and for their duals
\[
 (\omega_{r-q},\ \omega_{r-1}+\omega_{r-s}),\qquad
 (\omega_{r-s}+\omega_{r-q},\ \omega_{r-1}).
\]
\end{theorem}

\begin{proof}
Take $\boldsymbol\mu=(\omega_1,\omega_s,\omega_q)$. Conjecture~\ref{conj:BGL}
holds for $\boldsymbol\mu$ by Theorem~\ref{thm:three-factor-BGL}. We can cut
$\boldsymbol\mu$ either as $((\omega_1,\omega_s),(\omega_q))$ or as
$((\omega_1),(\omega_s,\omega_q))$. Every block has one factor, where there
is nothing to prove, or two fundamental factors, where
Corollary~\ref{cor:two-factor-symmetry} applies.
Proposition~\ref{prop:descent} gives the result for
$(\omega_1+\omega_s,\omega_q)$ and $(\omega_1,\omega_s+\omega_q)$. Finally,
since $\sigma$ is an involution, the identity in Conjecture~\ref{conj:BGL}
for $\boldsymbol\lambda$ is the same as the identity for
$\boldsymbol\lambda^*$, which gives the two dual sequences.
\end{proof}

\begin{corollary}[Reduction to fundamental weights]
\label{cor:reduction}
Let $G=\SL_r(\CC)$. If Conjecture~\ref{conj:BGL} holds for every finite
sequence of fundamental weights, then it holds for every finite sequence
of dominant weights.
\end{corollary}
\begin{proof}
We write each nonzero $\lambda_j$ as a sum of fundamental weights,
with repetitions, and let $\boldsymbol\mu^{(j)}$ be the resulting
sequence. Set \[
\boldsymbol\mu=
(\boldsymbol\mu^{(1)},\ldots,\boldsymbol\mu^{(m)}).
\]
By hypothesis, Conjecture~\ref{conj:BGL} holds for
$\boldsymbol\mu$ and for each block $\boldsymbol\mu^{(j)}$.
Proposition~\ref{prop:descent} therefore gives the conjecture
for $\boldsymbol\lambda$, since factors with $\lambda_j=0$ are trivial
representations and may be omitted.
\end{proof}

\subsection{Levi restrictions and multiplicity spaces}
\label{subsec:levi-multiplicity-spaces}

Let $V=\CC^r$ and
\[
 V(\boldsymbol\lambda)
 =\bigotimes_{j=1}^n\Lambda^{p_j}V,
 \qquad
 \boldsymbol\lambda=(\omega_{p_1},\ldots,\omega_{p_n}).
\]
For a tensor-product label
$\mathbf X=(A_1,\ldots,A_n)$, with $|A_j|=p_j$, set
\[
 c_i=\#\{j:i\in A_j\},\qquad
 \Omega_h=\{i\in[r]:c_i=h\},\qquad
 m_h=|\Omega_h|
 \quad(0\leq h\leq n).
\]
The vector $\mathbf c=(c_1,\ldots,c_r)$ is the weight profile:
\[
 \operatorname{wt}(\mathbf X)=\sum_{i=1}^r c_i\eps_i,
 \qquad
 \sum_{i=1}^r c_i=\sum_{j=1}^n p_j.
\]
If two such profiles define the same $\SL_r$-weight, their
difference is a constant vector. Since their coordinate sums agree,
the difference is zero. Thus
\[
 \operatorname{wt}(\mathbf X)=\operatorname{wt}(\mathbf X')
 \quad\Longleftrightarrow\quad
 \mathbf c(\mathbf X)=\mathbf c(\mathbf X').
\]

Now we fix a weight $\nu$ and its profile $\mathbf c$ and set
\[
 V_h=\operatorname{span}\{e_i:i\in\Omega_h\},
 \qquad \dim V_h=m_h,
\]
and let
\[
 M=\left(\prod_{\substack{0\leq h\leq n\\m_h>0}}
       \GL(V_h)\right)\cap\SL(V).
\]
This is a block-diagonal Levi subgroup, conjugate to a standard
Levi subgroup. Its derived subgroup is
\[
 M^{\mathrm{der}}
 =\prod_{\substack{0\leq h\leq n\\m_h>0}}\SL(V_h).
\]

\noindent If for each tensor factor and each block we set
\[
 a_{j,h}=|A_j\cap\Omega_h|,
\]
the resulting matrix $\mathbf a=(a_{j,h})$ satisfies
\begin{equation}
\label{eq:block-degree-relations}
 \begin{gathered}
 0\leq a_{j,h}\leq m_h,\qquad
 \sum_{h=0}^n a_{j,h}=p_j,\\
 \sum_{j=1}^n a_{j,h}=h\,m_h
 \qquad(0\leq h\leq n).
 \end{gathered}
\end{equation}
In particular, $a_{j,0}=0$ for every $j$. Let also $\mathfrak A(\mathbf c)$ denote the set of integer matrices
satisfying \eqref{eq:block-degree-relations} and  
\[
 \mathcal L_\nu
 =\{\mathbf X:\operatorname{wt}(\mathbf X)=\nu\},
\]
\[
 \mathcal L_\nu(\mathbf a)
 =\{\mathbf X\in\mathcal L_\nu:
       |A_j\cap\Omega_h|=a_{j,h}\ \text{for all }j,h\}.
\]
Then
\[
 \mathcal L_\nu
 =\bigsqcup_{\mathbf a\in\mathfrak A(\mathbf c)}
       \mathcal L_\nu(\mathbf a).
\]

\begin{lemma}[Decomposition of a weight space]
\label{lem:block-degree-decomposition}
For $\mathbf a\in\mathfrak A(\mathbf c)$, if
\[
 H=\{h\in\{1,\ldots,n\}:m_h>0\},
\]
then
\begin{equation}
\label{eq:block-degree-count}
 |\mathcal L_\nu(\mathbf a)|
 =
 \prod_{h\in H}
 \dim\left(
       \bigotimes_{j=1}^n\Lambda^{a_{j,h}}V_h
     \right)_{h\mathbf 1_{m_h}},
\end{equation}
where $\mathbf 1_{m_h}=(1,\ldots,1)$ and the subscript denotes
the $\GL(V_h)$-weight space of weight $h\mathbf 1_{m_h}$.
Consequently,
\begin{equation}
\label{eq:block-degree-weight-decomposition}
 V(\boldsymbol\lambda)_\nu
 \cong
 \bigoplus_{\mathbf a\in\mathfrak A(\mathbf c)}
 \bigotimes_{h\in H}
 \left(
       \bigotimes_{j=1}^n\Lambda^{a_{j,h}}V_h
 \right)_{h\mathbf 1_{m_h}}
\end{equation}
as vector spaces.
\end{lemma}

\begin{proof}
For each $j$, the decomposition of an exterior power relative to
$V=\bigoplus_{h=0}^n V_h$ is
\[
\left.\Lambda^{p_j}V\right|_{\prod_{h=0}^n\GL(V_h)}
\cong
\bigoplus_{\substack{0\leq d_h\leq m_h\\
                     \sum_{h=0}^n d_h=p_j}}
\boxtimes_{h=0}^n\Lambda^{d_h}V_h.
\]
Applying this to all tensor factors and rearranging the factors gives
\[
\begin{aligned}
V(\boldsymbol\lambda)
&=\bigotimes_{j=1}^n\Lambda^{p_j}V\\
&\cong
\bigotimes_{j=1}^n
\left(
\bigoplus_{\substack{0\leq a_{j,h}\leq m_h\\
                     \sum_{h=0}^n a_{j,h}=p_j}}
\boxtimes_{h=0}^n\Lambda^{a_{j,h}}V_h
\right)\\
&\cong
\bigoplus_{\substack{0\leq a_{j,h}\leq m_h\\
                     \sum_{h=0}^n a_{j,h}=p_j}}
\boxtimes_{h=0}^n
\left(\bigotimes_{j=1}^n\Lambda^{a_{j,h}}V_h\right).
\end{aligned}
\]

\noindent For $i\in\Omega_h$, we have $c_i=h$. Hence the restriction of
$\nu$ to the diagonal torus of $\GL(V_h)$ is
$h\mathbf 1_{m_h}$. Therefore
\[
\begin{aligned}
V(\boldsymbol\lambda)_\nu
\cong
\bigoplus_{\substack{0\leq a_{j,h}\leq m_h\\
                     \sum_h a_{j,h}=p_j}}
\bigotimes_{h=0}^n
\left(\bigotimes_{j=1}^n
\Lambda^{a_{j,h}}V_h\right)_{h\mathbf 1_{m_h}}.
\end{aligned}
\]
A summand can be nonzero only if
\[
\sum_{j=1}^n a_{j,h}=hm_h
\qquad (0\leq h\leq n).
\]
For $h=0$, this forces $a_{j,0}=0$ for every $j$, so that
\[
\bigotimes_{j=1}^n\Lambda^{a_{j,0}}V_0=\mathbb C.
\]
Consequently,
\[
V(\boldsymbol\lambda)_\nu
\cong
\bigoplus_{\mathbf a\in\mathfrak A(\mathbf c)}
\bigotimes_{h\in H}
\left(\bigotimes_{j=1}^n
\Lambda^{a_{j,h}}V_h\right)_{h\mathbf 1_{m_h}},
\tag{\ref{eq:block-degree-weight-decomposition}}
\]
where
\[
\mathfrak A(\mathbf c)=
\left\{
\mathbf a=(a_{j,h}):
\begin{array}{l}
0\leq a_{j,h}\leq m_h,\\[2pt]
\displaystyle\sum_{h=0}^n a_{j,h}=p_j
\quad(1\leq j\leq n),\\[5pt]
\displaystyle\sum_{j=1}^n a_{j,h}=hm_h
\quad(0\leq h\leq n)
\end{array}
\right\}.
\]

\noindent It remains to count the labels corresponding to a fixed
$\mathbf a\in\mathfrak A(\mathbf c)$.
For $B\subseteq\Omega_h$, write
\[
e_B=\bigwedge_{i\in B}e_i,
\]
with the indices in increasing order. The wedge tensor basis gives
\[
\begin{aligned}
\left(\bigotimes_{j=1}^n
\Lambda^{a_{j,h}}V_h\right)_{h\mathbf 1_{m_h}}
=
\operatorname{span}\Bigl\{
&e_{B_1}\otimes\cdots\otimes e_{B_n}:\\
&B_j\subseteq\Omega_h,\quad |B_j|=a_{j,h},\\
&\sum_{j=1}^n\mathbf 1_{B_j}
=h\mathbf 1_{\Omega_h}
\Bigr\}.
\end{aligned}
\]
Specifically,
\[
\dim\left(\bigotimes_{j=1}^n
\Lambda^{a_{j,h}}V_h\right)_{h\mathbf 1_{m_h}}
=
\#\left\{
(B_1,\ldots,B_n):
\begin{array}{l}
B_j\subseteq\Omega_h,\quad |B_j|=a_{j,h},\\
\displaystyle\sum_{j=1}^n\mathbf 1_{B_j}
=h\mathbf 1_{\Omega_h}
\end{array}
\right\}.
\]

\noindent For each $\mathbf a$, the map
\[
(A_1,\ldots,A_n)
\longmapsto
\bigl((A_j\cap\Omega_h)_{j=1}^n\bigr)_{h\in H}
\]
is a bijection from $\mathcal L_\nu(\mathbf a)$ onto the product,
over $h\in H$, of the sets counted above and its inverse is given by
\[
A_j=\bigsqcup_{h\in H}(A_j\cap\Omega_h),
\]
since $A_j\cap\Omega_0=\varnothing$. Thus
\[
\begin{aligned}
|\mathcal L_\nu(\mathbf a)|
&=
\prod_{h\in H}
\#\left\{
(B_1,\ldots,B_n):
\begin{array}{l}
B_j\subseteq\Omega_h,\quad |B_j|=a_{j,h},\\
\displaystyle\sum_{j=1}^n\mathbf 1_{B_j}
=h\mathbf 1_{\Omega_h}
\end{array}
\right\}\\
&=
\prod_{h\in H}
\dim\left(
\bigotimes_{j=1}^n\Lambda^{a_{j,h}}V_h
\right)_{h\mathbf 1_{m_h}}.
\end{aligned}
\tag{\ref{eq:block-degree-count}}
\]
This proves both assertions.
\end{proof}

\noindent Observe that for $n=2$, the matrix $\mathbf a$ is uniquely determined by the
profile. Indeed,
\[
\begin{gathered}
 a_{1,0}=a_{2,0}=0,\qquad
 a_{1,2}=a_{2,2}=m_2,\\
 a_{1,1}=p_1-m_2,\qquad
 a_{2,1}=p_2-m_2.
\end{gathered}
\]
Moreover, for every block $V_h$,
\[
 \Lambda^{a_{1,h}}V_h\otimes
 \Lambda^{a_{2,h}}V_h
\]
is multiplicity-free by Pieri's rule. Thus neither multiple
matrices nor repeated irreducible summands occur in this
two-factor restriction. Recall that for $n=2$, the corresponding basis uniqueness result is Theorem~\ref{thm:two-factor-uniqueness}.

\begin{example}
\label{ex:SL4-repeated-Levi-summands}
Let
\[
 G=\SL_4,\qquad
 \boldsymbol\lambda=(\omega_1,\omega_2,\omega_3),
 \qquad \nu=(1,1,2,2).
\]
Then
\[
 \Omega_1=\{1,2\},\qquad
 \Omega_2=\{3,4\},\qquad m_1=m_2=2,
\]
and
\[
 M^{\mathrm{der}}\cong\SL_2\times\SL_2.
\]
The matrices satisfying \eqref{eq:block-degree-relations} are
\[
 \begin{aligned}
 \mathbf a^{(1)}
 &=\begin{pmatrix}1&0\\0&2\\1&2\end{pmatrix},&
 \mathbf a^{(2)}
 &=\begin{pmatrix}0&1\\1&1\\1&2\end{pmatrix},&
 \mathbf a^{(3)}
 &=\begin{pmatrix}0&1\\0&2\\2&1\end{pmatrix}.
 \end{aligned}
\]
Rows correspond to the three tensor factors and columns to
$\Omega_1,\Omega_2$. Let $\varpi$ be the fundamental weight of $\SL_2$. As
$\SL_2$-modules,
\[
 \Lambda^0\CC^2\cong\Lambda^2\CC^2\cong V_{\SL_2}(0),
 \qquad
 \Lambda^1\CC^2\cong V_{\SL_2}(\varpi),
\]
and
\[
 V_{\SL_2}(\varpi)\otimes V_{\SL_2}(\varpi)
 \cong
 V_{\SL_2}(2\varpi)\oplus V_{\SL_2}(0).
\]
The multiplicities of the irreducible $M^{\mathrm{der}}$-modules
in the summands corresponding to the three matrices are
\[
\begin{array}{c|ccc|c}
 &\mathbf a^{(1)}&\mathbf a^{(2)}&\mathbf a^{(3)}
 &\text{total}\\ \hline
 V_{\SL_2}(2\varpi)\boxtimes V_{\SL_2}(2\varpi)
 &0&1&0&1\\
 V_{\SL_2}(2\varpi)\boxtimes V_{\SL_2}(0)
 &1&1&0&2\\
 V_{\SL_2}(0)\boxtimes V_{\SL_2}(2\varpi)
 &0&1&1&2\\
 V_{\SL_2}(0)\boxtimes V_{\SL_2}(0)
 &1&1&1&3
\end{array}
\]
The central character of each block is fixed by
$\sum_j a_{j,h}=hm_h$, so the repeated derived-Levi types
also have the same central character. Each displayed irreducible module has a one-dimensional
zero-weight space. Hence
\[
 |\mathcal L_\nu(\mathbf a^{(1)})|=2,\qquad
 |\mathcal L_\nu(\mathbf a^{(2)})|=4,\qquad
 |\mathcal L_\nu(\mathbf a^{(3)})|=2,
\]
and
\[
 \dim V(\boldsymbol\lambda)_\nu=2+4+2=8.
\]
In particular,
\[
 V_{\SL_2}(0)\boxtimes V_{\SL_2}(0)
\]
occurs once in each of the three summands.
\end{example}

\begin{example}[Repeated summands within a single matrix]
\label{ex:SL3-repeated-Levi-summand}
Let
\[
 G=\SL_3,\qquad
 \boldsymbol\lambda=(\omega_1,\omega_1,\omega_1),
 \qquad \nu=0.
\]
There is a single nonempty block,
\[
 \Omega_1=\{1,2,3\},\qquad m_1=3,
\]
and a single matrix,
\[
 \mathbf a=\begin{pmatrix}1\\1\\1\end{pmatrix}.
\]
Nevertheless, the corresponding $\SL_3$-module decomposes as
\[
 V(\omega_1)^{\otimes3}
 \cong
 V(3\omega_1)\oplus
 2V(\omega_1+\omega_2)\oplus V(0).
\]
Since
\[
 \dim V(3\omega_1)_0=1,\qquad
 \dim V(\omega_1+\omega_2)_0=2,\qquad
 \dim V(0)_0=1,
\]
we obtain
\[
 \dim V(\boldsymbol\lambda)_0
 =1+2\cdot2+1=6.
\]
Thus a repeated irreducible summand can also occur within
a single matrix.
\end{example}

\noindent We now study the following case. For labels $Y,X$ of the same weight, we write
\[
 Y\preceq X
 \quad\Longleftrightarrow\quad
 \Delta_k(Y,X)\in Q_+
 \qquad(1\leq k<n).
\]
Let $F_X$ be the tensor-product MV basis and $E_Z$ the MV basis
of $V(\boldsymbol\lambda)_\nu$. The transition matrix is unitriangular with respect to the order
$\preceq$, i.e.,
\begin{equation}
\label{eq:higher-factor-unitriangularity}
F_X=\sum_{Y\preceq X}a_{Y,X}E_Y,
\qquad a_{X,X}=1.
\end{equation}
The other condition in the two-factor characterization is
\[
 V_{J,\leq_J\mu}\cap V(\boldsymbol\lambda)_\nu
 \quad\text{is spanned by a subset of }\{E_Z\}.
\]

\begin{proposition}[Failure of uniqueness]
\label{prop:failure-of-P1-uniqueness}
Suppose that distinct labels $Y\prec X$ satisfy
\[
 \eta_J(Y)\leq_J\eta_J(X)
 \qquad\text{for every }M_J.
\]
Then, for every $t\in\CC$, the basis
\[
 \widetilde E_X=E_X+tE_Y,\qquad
 \widetilde E_Z=E_Z\quad(Z\ne X)
\]
satisfies both \eqref{eq:higher-factor-unitriangularity}
and the $P1$ condition. In particular, for $t\ne0$ it is
different from the MV basis.
\end{proposition}

\begin{proof}
Since $Y\prec X$, the change of basis is invertible. From
\eqref{eq:higher-factor-unitriangularity},
\[
\begin{aligned}
 F_W
 &=\sum_{Z\notin\{X,Y\}}a_{Z,W}E_Z
   +a_{X,W}\widetilde E_X
   +(a_{Y,W}-t a_{X,W})\widetilde E_Y.
\end{aligned}
\]
If $a_{X,W}\ne0$, then
\[
 Y\prec X\preceq W
 \quad\Longrightarrow\quad
 Y\prec W.
\]
Thus the support condition is preserved, and the diagonal
coefficients remain equal to $1$. For the $P1$ condition, Lemma~\ref{lem:P1-crystal-membership}
gives
\[
 V_{J,\leq_J\mu}\cap V(\boldsymbol\lambda)_\nu
 =
 \operatorname{span}
 \{E_Z:\eta_J(Z)\leq_J\mu\}.
\]
The hypothesis
$\eta_J(Y)\leq_J\eta_J(X)$ implies
\[
 \eta_J(X)\leq_J\mu
 \quad\Longrightarrow\quad
 \eta_J(Y)\leq_J\mu.
\]
Hence replacing $E_X$ by $E_X+tE_Y$ does not change the
span of any of these subspaces, i.e., both conditions are preserved.
\end{proof}
\begin{example}[Failure of uniqueness from unitriangularity and (P1)]
\label{ex:SL4-P1-nonuniqueness}
Let
\[
 \boldsymbol\lambda=(\omega_1,\omega_3,\omega_2)
 \quad\text{in }\SL_4,
\]
and consider the pair $Y,X$ from Remark~\ref{rem:scope}.
As shown there, $Y$ and $X$ have the same weight,
\[
 \Delta_1(Y,X)=\alpha_1,\qquad
 \Delta_2(Y,X)=\alpha_1+\alpha_2+\alpha_3,
\]
and
\[
 \eta_J(Y)\leq_J\eta_J(X)
 \qquad\text{for every }M_J.
\]
In particular, $(Y,X)$ is a distinct admissible pair, therefore
Proposition~\ref{prop:failure-of-P1-uniqueness} shows
that the unitriangularity condition together with all the
$P1$ conditions does not characterize the MV basis uniquely
for this triple.
\end{example}

\noindent On the other hand, for the family
\[
 \boldsymbol\lambda=(\omega_1,\omega_s,\omega_q),
 \qquad
 1\leq s\leq\min\{2,r-1\},\quad 1\leq q\leq r-1,
\]
Theorem~\ref{thm:three-factor-H} gives the converse
restriction on distinct admissible pairs satisfying
\eqref{eq:Levi-highest-weight-order}: for every such pair,
\[
 \Delta_1(Y,X)=0
 \qquad\text{or}\qquad
 \Delta_2(Y,X)=0.
\]
If $\Delta_k(Y,X)=0$, the factorization lemma gives
\[
 a^{\boldsymbol\lambda}_{Y,X}
 =
 a^{\boldsymbol\lambda_{\leq k}}_{Y_{\leq k},X_{\leq k}}
 a^{\boldsymbol\lambda_{>k}}_{Y_{>k},X_{>k}}.
\]
For the pair in Remark~\ref{rem:scope}, neither partial difference
vanishes, so this factorization cannot be applied.

\paragraph{Multiplicity spaces.}
For $V=V(\boldsymbol\lambda)$ and a standard Levi subgroup $M_J$,
we write
\[
 V_{J,\mu}
 \cong V_J(\mu)\otimes\mathcal M^{\boldsymbol\lambda}_{J,\mu},
 \qquad
 \mathcal M^{\boldsymbol\lambda}_{J,\mu}
 :=
 \operatorname{Hom}_{M_J}(V_J(\mu),V_{J,\mu}).
\]
Thus
\[
 \dim\mathcal M^{\boldsymbol\lambda}_{J,\mu}
 =
 \operatorname{mult}_{M_J}
 \bigl(V_J(\mu),V(\boldsymbol\lambda)\bigr).
\]
Condition \textup{(P2)} gives a decomposition of
$V_{J,\mu}$ into irreducible $M_J$-summands compatible with
the induced basis. It does not, however, specify how the
Hodge--Sch\"utzenberger correspondence matches these summands
with those in the corresponding isotypic component for
$\boldsymbol\lambda^*$. Writing $\mu^\dagger=-w_0^J\mu$, when
$\dim\mathcal M^{\boldsymbol\lambda}_{J,\mu}>1$,
the remaining problem is to determine the matrix induced by
the Hodge--Sch\"utzenberger correspondence between
\[
 \mathcal M^{\boldsymbol\lambda}_{J,\mu}
 \quad\text{and}\quad
 \mathcal M^{\boldsymbol\lambda^*}_{J,\mu^\dagger},
\]
with respect to the multiplicity bases determined by the
two MV bases. The examples above exhibit both sources of repeated copies:
the same irreducible Levi module can occur for distinct
matrices $\mathbf a$, or more than once for a single matrix.

\appendix

\section{Subsequence inequalities for three factors}
\label{app:subsequence-inequalities}
For an interval $I\subseteq[r]$ and a label $Z=(A_Z,B_Z,C_Z)$, let
\[
 w_Z(I)=A_Z(I)^\downarrow B_Z(I)^\downarrow C_Z(I)^\downarrow .
\]
Let $\operatorname{dec}_Z(I)$ be the length of the longest strictly decreasing subsequence of
$w_Z(I)$, and let $\operatorname{inc}_Z(I)$ be $1$ if $w_Z(I)$ has a weakly increasing subsequence of length three using one letter from each of its three blocks, and $0$ otherwise.

\begin{lemma}[Subsequence comparison for $1\leq s\leq2$]
\label{lem:mixed-subsequence-comparison}
Let $A,A',B,B',C,C'\subseteq[r]$ satisfy
\[
 |A|=|A'|=1,\quad 1\leq|B|=|B'|=s\leq\min(2,r-1),\quad
 |C|=|C'|=q,\quad 1\leq q\leq r-1,
\]
and $\mathbf1_A+\mathbf1_B+\mathbf1_C
=\mathbf1_{A'}+\mathbf1_{B'}+\mathbf1_{C'}$. Let
\[
 \Delta_1^{(t)}(Y,X)
 =|A'\cap[1,t]|-|A\cap[1,t]|,\quad
 \Delta_2^{(t)}(Y,X)
 =|A'\cap[1,t]|+|B'\cap[1,t]|-|A\cap[1,t]|-|B\cap[1,t]|.
\]
If $\Delta_1^{(t)}(Y,X),\Delta_2^{(t)}(Y,X)\geq0$ for all $t$,
$\Delta_1^{(t)}(Y,X)\not\equiv0$ and
$\Delta_2^{(t)}(Y,X)\not\equiv0$, then some interval
$I=[u,v]$ with $1\leq u<v\leq r$ satisfies
\[
 \operatorname{dec}_Y(I)<\operatorname{dec}_X(I)
 \qquad\text{or}\qquad
 \operatorname{inc}_Y(I)>\operatorname{inc}_X(I),
\]
where $X=(A,B,C)$ and $Y=(A',B',C')$.
\end{lemma}

\begin{proof}
For a label $Z=(A_Z,B_Z,C_Z)$ and an interval $I$, let
$w_Z(I)=A_Z(I)^{\downarrow}B_Z(I)^{\downarrow}C_Z(I)^{\downarrow}$, where
$S(I)=S\cap I$ and $\downarrow$ lists a set in decreasing order.
Let $\operatorname{dec}_Z(I)$ be the length of the longest \emph{strictly
decreasing} subsequence of $w_Z(I)$, and let $\operatorname{inc}_Z(I)\in\{0,1\}$
be $1$ if and only if $w_Z(I)$ has a weakly increasing subsequence of
length three (necessarily one letter from each block). Write $A=\{a\}$,
$A'=\{a'\}$, $X=(A,B,C)$, $Y=(A',B',C')$. The hypotheses on $\Delta_1^{(t)}$ say precisely that $a'<a$. For multisets
$M,N$ of integers of the same size write $M\preceq N$ if the $i$-th smallest
element of $M$ is at most that of $N$ for every $i$ (equivalently
$\#\{m\in M:m\le t\}\ge\#\{n\in N:n\le t\}$ for all $t$). The hypotheses on
$\Delta_2^{(t)}$ say
\begin{equation}\tag{D}
 \{a'\}\cup B'\ \preceq\ \{a\}\cup B,\qquad
 \{a'\}\cup B'\ \neq\ \{a\}\cup B\ \text{ (as multisets)}.
\end{equation}
The profile identity reads
\begin{equation}\tag{P}
 \mathbf1_{C'}-\mathbf1_{C}
 =\mathbf1_{A}+\mathbf1_{B}-\mathbf1_{A'}-\mathbf1_{B'},
\end{equation}
and its left-hand side takes values in $\{-1,0,1\}$ at every point. In
particular, a point where the right-hand side equals $1$ (resp.\ $-1$) lies in
$C'\setminus C$ (resp.\ $C\setminus C'$).

\begin{lemma}\label{lem:sc-A}
Let $I=[a',a]$ and suppose $B\cap I=B'\cap I$. Then
$\operatorname{dec}_Y(I)<\operatorname{dec}_X(I)$ or
$\operatorname{inc}_Y(I)>\operatorname{inc}_X(I)$.
\end{lemma}
\begin{proof}
Let $S=B\cap I$. By (P), $\mathbf1_{C'}-\mathbf1_C=\mathbf1_a-\mathbf1_{a'}$ on
$I$. Hence, with $E=C\cap C'\cap I\subseteq(a',a)$, we have
$C\cap I=E\cup\{a'\}$ and $C'\cap I=E\cup\{a\}$.
If $S\ne\varnothing$ and $b\in S$, then $a'\le b\le a$ with $a'\in A'$,
$b\in B'$, $a\in C'$ gives $\operatorname{inc}_Y(I)=1$. A triple for $X$
would need $a\le b\le c$ with $c\in C\cap I\subseteq[a',a)$, which is
impossible; so $\operatorname{inc}_X(I)=0$.
If $S=\varnothing$, then $w_X(I)=a\,(E\cup\{a'\})^\downarrow$ is strictly
decreasing, so $\operatorname{dec}_X(I)=|E|+2$, while
$w_Y(I)=a'\,(E\cup\{a\})^\downarrow$, in which $a'$ is smaller than every
later letter, so $\operatorname{dec}_Y(I)=|E|+1$.
\end{proof}

\begin{lemma}[First triple criterion]\label{lem:sc-T0}
If $a\notin B'$ and $B'\cap[a',a)\ne\varnothing$, then
$\operatorname{inc}_Y([a',a])>\operatorname{inc}_X([a',a])$.
\end{lemma}
\begin{proof}
At $a$, the right-hand side of (P) equals $1+\mathbf1_B(a)$, hence
$a\notin B$, $a\in C'$, $a\notin C$. If $b'\in B'\cap[a',a)$, then
$a'\le b'\le a$ gives $\operatorname{inc}_Y=1$. A triple for $X$ inside
$[a',a]$ would be $a\le b\le c\le a$, forcing $a\in B$; hence
$\operatorname{inc}_X=0$.
\end{proof}

\begin{lemma}[Second triple criterion]\label{lem:sc-T2}
Let $y\in C'\setminus C$ with $y\ge a$, suppose $B\cap[a,y)=\varnothing$ and
that there is $b'\in B'\cap[a',y]$. Then
$\operatorname{inc}_Y([a',y])>\operatorname{inc}_X([a',y])$.
\end{lemma}
\begin{proof}
The triple $a'\le b'\le y$ gives $\operatorname{inc}_Y=1$. A triple
$a\le b\le c$ for $X$ in $[a',y]$ has $b\in B\cap[a,y]$, hence $b=y$, or
$c=y\in C$, a contradiction.
\end{proof}

\begin{lemma}\label{lem:sc-F}
Suppose $a\in B'$, $B\cap[a',a]=\varnothing$, and $a$ is the only element of
$B'$ in $[a',a]$. Then the conclusion of
Lemma~\ref{lem:mixed-subsequence-comparison} holds for $I=[a',a]$.
\end{lemma}
\begin{proof}
By (P) at $a$ we have $a\in C\iff a\in C'$, and at $a'$ we have
$a'\in C\setminus C'$. If $a\in C'$, then $a'\le a\le a$ gives
$\operatorname{inc}_Y(I)=1$, while $\operatorname{inc}_X(I)=0$ since
$B\cap I=\varnothing$. If $a\notin C'$, put
$E=C\cap(a',a)=C'\cap(a',a)$. Then $w_X(I)=a\,(E\cup\{a'\})^\downarrow$ and
$w_Y(I)=a'\,a\,E^\downarrow$, so
$\operatorname{dec}_X(I)=|E|+2$ and $\operatorname{dec}_Y(I)=|E|+1$.
\end{proof}

\medskip\noindent We need now to derive a formula for the longest decreasing subsequence $\operatorname{dec}$. For an interval $I$ let $E=C\cap C'\cap I$ and, for $Z\in\{X,Y\}$,
$C_Z^\circ=(C_Z\cap I)\setminus E$, where $C_X=C$, $C_Y=C'$.
A \emph{head} $\sigma$ of $Z$ is a (possibly empty) strictly decreasing
subsequence of $A_Z(I)^\downarrow B_Z(I)^\downarrow$. Let $\min\sigma$ be its
last letter ($\infty$ if $\sigma=\varnothing$) and
\[
 \operatorname{sc}(\sigma)=|\sigma|+\#\{c\in C_Z^\circ:\ c<\min\sigma\}.
\]
A strictly decreasing subsequence of $w_Z(I)$ consists of a head $\sigma$
followed by any letters of the third block below $\min\sigma$ (the third block
is itself decreasing). Hence
\begin{equation}\label{eq:sc-head-formula}
 \operatorname{dec}_Z(I)=\max_{\sigma}\Bigl(
 \operatorname{sc}(\sigma)+|E|-\#\{e\in E:\ e\ge\min\sigma\}\Bigr).
\end{equation}

\begin{lemma}[Head domination]\label{lem:sc-head}
Suppose every head $\sigma$ of $Y$ admits a head $\tau$ of $X$ with
$\operatorname{sc}(\tau)\ge\operatorname{sc}(\sigma)+1$ and either
\textup{(a)} $\min\tau\ge\min\sigma$, or
\textup{(b)} $\min\tau=\max I$ and $\max I\in C_X^\circ\cup C_Y^\circ$.
Then $\operatorname{dec}_X(I)\ge\operatorname{dec}_Y(I)+1$, whatever the set
$E$ (and the rest of $I$) is.
\end{lemma}
\begin{proof}
Let $\sigma$ realize $\operatorname{dec}_Y(I)$ in
\eqref{eq:sc-head-formula} and take $\tau$ as in the hypothesis.
In case (a), $\#\{e\in E:e\ge\min\tau\}\le\#\{e\in E:e\ge\min\sigma\}$, so the
value of $\tau$ in \eqref{eq:sc-head-formula} is at least that of $\sigma$ plus
one. In case (b), $\max I$ lies in exactly one of $C,C'$, hence not in $E$,
so $\#\{e\in E:e\ge\min\tau\}=0$ and the value of $\tau$ is at least
$\operatorname{sc}(\tau)+|E|\ge$ the value of $\sigma$ plus one.
\end{proof}

\noindent Table~\ref{tab:sc-certificates} lists ten local configurations
$\mathsf C_1,\dots,\mathsf C_{10}$. The letters $x<y<z<w$ are elements of $I$,
the smallest equal to $\min I$ and the largest to $\max I$; each side is
recorded as $(A_Z\cap I\mid B_Z\cap I\mid C_Z^\circ)$, the set $E$ and all
other points of $I$ being arbitrary. The last column assigns to every head of
$Y$ a head of $X$ as in Lemma~\ref{lem:sc-head} (scores in parentheses).
Thus $\operatorname{dec}_Y(I)<\operatorname{dec}_X(I)$ in each of the ten
configurations.

\begin{table}[ht]
\centering\small
\renewcommand{\arraystretch}{1.25}
\begin{tabular}{@{}llll@{}}
\toprule
 & $X$ & $Y$ & heads of $Y$ $\to$ head of $X$\\
\midrule
$\mathsf C_1$ & $(-\mid y\mid x)$ & $(x\mid y\mid -)$ &
 $\varnothing(0)\to\varnothing(1)$;\ $x,y(1)\to y(2)$\\
$\mathsf C_2$ & $(-\mid -\mid xy)$ & $(x\mid y\mid -)$ &
 $\varnothing(0),x(1),y(1)\to\varnothing(2)$\\
$\mathsf C_3$ & $(-\mid y\mid x)$ & $(-\mid x\mid y)$ &
 $\varnothing,x(1)\to y(2)$\\
$\mathsf C_4$ & $(-\mid y\mid x)$ & $(x\mid -\mid y)$ &
 $\varnothing,x(1)\to y(2)$\\
$\mathsf C_5$ & $(-\mid xy\mid x)$ & $(x\mid x\mid y)$ &
 $\varnothing,x(1)\to y(2)$\\
$\mathsf C_6$ & $(-\mid yz\mid x)$ & $(y\mid x\mid z)$ &
 $\varnothing,x,y(1)\to z(2)$;\ $yx(2)\to zy(3)$\\
$\mathsf C_7$ & $(z\mid y\mid x)$ & $(y\mid x\mid z)$ &
 $\varnothing,x,y(1)\to z(2)$;\ $yx(2)\to zy(3)$\\
$\mathsf C_8$ & $(z\mid y\mid x)$ & $(y\mid xz\mid -)$ &
 $\varnothing(0)\to\varnothing(1)$;\ $x,y,z(1)\to z(2)$;\
 $yx,zx(2)\to zy(3)$\\
$\mathsf C_9$ & $(z\mid yz\mid x)$ & $(y\mid xz\mid z)$ &
 $\varnothing,x,y,z(1)\to z(2)$;\ $yx,zx(2)\to zy(3)$\\
$\mathsf C_{10}$ & $(w\mid yz\mid x)$ & $(z\mid xy\mid w)$ &
 $\varnothing,x,y,z(1)\to w(2)$;\ $yx,zx,zy(2)\to wz(3)$;\
 $zyx(3)\to wzy(4)$\\
\bottomrule
\end{tabular}
\caption{Local certificates; $xy$ denotes the set $\{x,y\}$.}
\label{tab:sc-certificates}
\end{table}

\noindent\emph{Example: $\mathsf C_{10}$.}
Here $I=[x,w]$, $w_X(I)=w\,z\,y\,(E\cup\{x\})^\downarrow$ and
$w_Y(I)=z\,y\,x\,(E\cup\{w\})^\downarrow$, with $C_X^\circ=\{x\}$,
$C_Y^\circ=\{w\}$. The heads of $Y$ are the strictly decreasing subsequences of
$z\,y\,x$ (namely $\varnothing,z,y,x,zy,zx,yx,zyx$); their scores are
$1,1,1,1,2,2,2,3$ (the empty head scores $1$ because $w\in C_Y^\circ$).
The heads $\varnothing,x,y,z$ are matched with the head $w$ of $X$ (score
$|{w}|+\#\{x<w\}=2$), via case (b) of Lemma~\ref{lem:sc-head} since
$\min\tau=w=\max I\in C_Y^\circ$. The heads $yx,zx,zy$ (minima $x,x,y$) are
matched with $wz$ (score $3$, minimum $z$), and $zyx$ (minimum $x$) with $wzy$
(score $4$, minimum $y$), both via case (a).

\medskip\noindent We now deal with the case $s=1$. We write $B=\{b\}$, $B'=\{b'\}$. The sets $A_Z\cap I$, $B_Z\cap I$ are
given by the listed points, and $C_X^\circ$, $C_Y^\circ$ are the points of $I$
where the right-hand side of (P) equals $-1$, respectively $+1$.

\smallskip\noindent\emph{(1) $a'\le b'<a$.}
Lemma~\ref{lem:sc-T0}.

\smallskip\noindent\emph{(2) $b'=a$.} By (D), $b\ge a'$ and $b\ne a'$.
If $b<a$, the interval $[a',b]$ is of type $\mathsf C_4$ with $(x,y)=(a',b)$.
If $b=a$, we use Lemma~\ref{lem:sc-A}; if $b>a$, we use Lemma~\ref{lem:sc-F}.

\smallskip\noindent\emph{(3) $b'<a'$ or $b'>a$.}
If $b'>a$ then $b\ge b'$ by (D). Hence if $b=b'$, or $b'>a$, or $b>a$, we have
$B\cap[a',a]=B'\cap[a',a]$ and Lemma~\ref{lem:sc-A} applies. Otherwise
$b'<a'$ and, by (D), $b'<b\le a$; moreover $b\ne a$, since (P) cannot take the
value $2$. If $b<a'$, the interval $[b',b]$ is of type $\mathsf C_3$ with
$(x,y)=(b',b)$. If $b=a'$, then $[b',a]$ is of type $\mathsf C_7$ with
$(x,y,z)=(b',a',a)$. If $a'<b<a$, then $[a',b]$ is of type $\mathsf C_4$.

\medskip\noindent The next case is $s=2$. We write $B=\{b_1<b_2\}$ and $I_0=[a',a]$. If $a\notin B'$ and
$B'\cap[a',a)\ne\varnothing$, Lemma~\ref{lem:sc-T0} applies. Otherwise
either $a\in B'$ (Case~I) or $B'\cap I_0=\varnothing$ (Case~II).
In the lists below, ``type $\mathsf C_k$ on $[p,q]$ with $(x,y,\dots)=(\dots)$''
means that the configuration restricted to $[p,q]$ is the one of
Table~\ref{tab:sc-certificates} under the indicated substitution.

\smallskip\noindent\emph{Case I: $a\in B'$.} We write $B'=\{a,\beta\}$.
Cancelling $a$ in (D) and (P) gives
\begin{equation}\label{eq:sc-caseI}
 \{a',\beta\}\preceq\{b_1,b_2\},\quad
 \{a',\beta\}\ne\{b_1,b_2\},\quad
 \mathbf1_{C'}-\mathbf1_C=\mathbf1_{b_1}+\mathbf1_{b_2}-\mathbf1_{a'}-\mathbf1_{\beta}.
\end{equation}

\emph{I.1: $\beta>a$.} Then $b_1\ge a'$ and $b_2\ge\beta$. If $b_1>a$, we use
Lemma~\ref{lem:sc-F}; if $b_1=a$, we use Lemma~\ref{lem:sc-A}. Let $b_1<a$. If
$b_2>\beta$, we use Lemma~\ref{lem:sc-T2} with $y=b_2$. If $b_2=\beta$, then
$b_1>a'$ and $[a',b_1]$ is of type $\mathsf C_4$ with $(x,y)=(a',b_1)$.

\emph{I.2: $a'<\beta<a$.} Then $b_1\ge a'$ and $b_2\ge\beta$.
\begin{itemize}[nosep]
\item $b_1=a'$ (so $b_2>\beta$): if $b_2<a$, $[\beta,b_2]$ is of type
 $\mathsf C_3$; if $b_2\ge a$, we use Lemma~\ref{lem:sc-T2} with $y=b_2$.
\item $a'<b_1<\beta$: $[a',b_1]$ is of type $\mathsf C_4$.
\item $b_1=\beta$: $[a',\beta]$ is of type $\mathsf C_1$ with $(x,y)=(a',\beta)$.
\item $\beta<b_1<a$: $[\beta,b_1]$ is of type $\mathsf C_3$.
\item $b_1\ge a$: $[a',\beta]$ is of type $\mathsf C_2$ with $(x,y)=(a',\beta)$.
\end{itemize}

\emph{I.3: $\beta=a'$.} Here $b_1\ge a'$, and the right-hand side of
\eqref{eq:sc-caseI} is at least $-1$ at $a'$, so $b_1=a'$. If $b_2<a$,
$[a',b_2]$ is of type $\mathsf C_5$ with $(x,y)=(a',b_2)$; if $b_2=a$, use
Lemma~\ref{lem:sc-A}; if $b_2>a$, we use Lemma~\ref{lem:sc-T2} with $y=b_2$.

\emph{I.4: $\beta<a'$.} Then $b_1\ge\beta$ and $b_2\ge a'$. If
$B\cap I_0=\varnothing$, use Lemma~\ref{lem:sc-F}; if $B\cap I_0=\{a\}$, we use
Lemma~\ref{lem:sc-A}. In the remaining cases:
\begin{itemize}[nosep]
\item $b_1=\beta$: then $a'<b_2<a$ and $[a',b_2]$ is of type $\mathsf C_4$.
\item $\beta<b_1<a'$: $[\beta,b_1]$ is of type $\mathsf C_3$.
\item $b_1=a'$ and $b_2<a$: $[\beta,b_2]$ is of type $\mathsf C_6$ with
 $(x,y,z)=(\beta,a',b_2)$.
\item $b_1=a'$ and $b_2=a$: $[\beta,a]$ is of type $\mathsf C_9$ with
 $(x,y,z)=(\beta,a',a)$.
\item $b_1=a'$ and $b_2>a$: $[\beta,a]$ is of type $\mathsf C_8$ with
 $(x,y,z)=(\beta,a',a)$.
\item $a'<b_1<a$: $[a',b_1]$ is of type $\mathsf C_4$.
\end{itemize}
(The values $b_1\ge a$ give $B\cap I_0=\varnothing$ or $\{a\}$.)

\smallskip\noindent\emph{Case II: $B'\cap I_0=\varnothing$.}
If $B\cap I_0=\varnothing$, we use Lemma~\ref{lem:sc-A}. Assume
$B\cap I_0\ne\varnothing$. Since $a\notin B'$, (P) gives $a\notin B$, so
$B\cap[a',a)\ne\varnothing$. Write $B'=\{\beta_1<\beta_2\}$.

\emph{II.0: $\beta_1,\beta_2>a$.} Impossible: $\{a'\}\cup B'$ has exactly one
element $\le a$, while $\{a\}\cup B$ has at least two, contradicting (D).

\emph{II.1: $\beta_1<a'<a<\beta_2$.} By (D), $b_1\in[a',a)$, $b_2>a$.
If $b_1>a'$, $[a',b_1]$ is of type $\mathsf C_4$. If $b_1=a'$, $[\beta_1,a]$ is
of type $\mathsf C_7$ with $(x,y,z)=(\beta_1,a',a)$.

\emph{II.2: $\beta_1<\beta_2<a'$.} By (D), $b_1\ge\beta_1$ and $b_2\ge\beta_2$.
\begin{itemize}[nosep]
\item $a'<b_1<a$: $[a',b_1]$ is of type $\mathsf C_4$.
\item $b_1=a'$ and $b_2<a$: $[\beta_2,b_2]$ is of type $\mathsf C_6$ with
 $(x,y,z)=(\beta_2,a',b_2)$.
\item $b_1=a'$ and $b_2>a$: $[\beta_2,a]$ is of type $\mathsf C_7$ with
 $(x,y,z)=(\beta_2,a',a)$.
\end{itemize}
Otherwise $b_1<a'$ and $a'\le b_2<a$ (note that $b_1\ge a$ would give
$B\cap[a',a)=\varnothing$).
\begin{itemize}[nosep]
\item $b_1>\beta_2$: $[\beta_2,b_1]$ is of type $\mathsf C_3$.
\item $b_1\le\beta_2$ and $b_2>a'$: $[a',b_2]$ is of type $\mathsf C_4$.
\item $b_1<\beta_2$ and $b_2=a'$: $[\beta_2,a]$ is of type $\mathsf C_7$ with
 $(x,y,z)=(\beta_2,a',a)$.
\item $b_1=\beta_2$ and $b_2=a'$: $[\beta_1,a]$ is of type $\mathsf C_{10}$
 with $(x,y,z,w)=(\beta_1,\beta_2,a',a)$.
\end{itemize}

\end{proof}

\end{document}